\documentclass[12pt,a4paper]{amsart}
\usepackage{amsmath,amsthm,amsfonts,relsize,units,amssymb,setspace}

\usepackage[colorlinks]{hyperref}

\DeclareMathOperator\Length{\mathrm{Length}}
\DeclareMathOperator\ad{ad}

\DeclareMathOperator\R{\mathbb R}
\DeclareMathOperator\lie{Lie}

\DeclareMathOperator\tr{\mathsf{tr}}
\DeclareMathOperator\dist{\mathrm{dist}}

\DeclareMathOperator\Exp{Exp}
\DeclareMathOperator\V{\mathsf{V}}
\DeclareMathOperator\der{{\texttt{Der}\,}(\V)}
\DeclareMathOperator\str{{\texttt{str}}(\V)}
\DeclareMathOperator\Str{{\texttt{Str}}(\V)}
\DeclareMathOperator\glv{{\mathsf{GL}}(\V)}
\DeclareMathOperator\Aut{{\texttt{Aut}}(\V)}

\DeclareMathOperator\GO{{\mathsf{G}}(\Omega)}
\DeclareMathOperator\bv{{\mathsf{B}}(\V)}
\DeclareMathOperator\B{\mathsf{B}}
\DeclareMathOperator\h{\mathsf{H}}
\DeclareMathOperator\bh{\mathsf{B}(\h)}
\DeclareMathOperator\Tr{\mathsf{Tr}}

\begin{document}

\newtheorem*{theorem*}{Theorem}
\newtheorem{teo}{Theorem}[section]
\theoremstyle{definition}
\newtheorem{prop}[teo]{Proposition}
\newtheorem{lema}[teo]{Lemma}
\newtheorem{coro}[teo]{Corollary}
\newtheorem{defi}[teo]{Definition}
\newtheorem{rem}[teo]{Remark}
\newtheorem{ejem}[teo]{Example}
\newtheorem{problem}[teo]{Problem}
\newtheorem{conj}[teo]{Conjecture}

\markboth{}{}

\makeatletter

\title[Finsler geometry of the structure group]{\vspace*{-2cm}Connections and Finsler geometry of the structure group of a JB-algebra}
\date{}
\author{Gabriel Larotonda}
\address{Departamento de Matem\'atica, Facultad de Cs. Exactas y Naturales, Universidad de Buenos Aires \& Instituto Argentino de Matem\'atica (CONICET), Argentina}
\email{glaroton@dm.uba.ar}
\author{Jos\'e Luna}
\address{Instituto Argentino de Matem\'atica ``Alberto P. Calder\'on'' (CONICET), Buenos Aires, Argentina}
\email{jluna@dm.uba.ar,jaleluna@gmail.com}

\makeatother

\begin{abstract}{We endow the Banach-Lie structure group $Str(V)$ of an infinite dimensional JB-algebra $V$ with a left-invariant connection and Finsler metric, and we compute all the quantities of its connection. We show how this connection reduces to $G(\Omega)$, the group of transformations that preserve the positive cone $\Omega$ of the algebra $V$, and to $Aut(V)$, the group of Jordan automorphisms of the algebra. We present the cone $\Omega$ as a homogeneous space for the action of $G(\Omega)$, therefore inducing a quotient Finsler metric and distance. With the techniques introduced, we prove the minimality of the one-parameter groups in $\Omega$ for any symmetric gauge norm in $V$. We establish that the two presentations of the Finsler metric in $\Omega$ give the same distance there, which helps us prove the minimality of certain paths in $G(\Omega)$ for its left-invariant Finsler metric.}
\end{abstract}

\maketitle

\setlength{\parindent}{0cm} 

\thispagestyle{empty}


\section{Introduction}

The geometry of the general linear group with a left-invariant Riemannian metric has been a subject of interest since the seminal works of Arnol'd on the group of diffeomorphisms preserving the volume of a spatial region \cite{arnold}. In his construction, the connection carries more relevant structure than the distance, and it is by solving Euler's equation for geodesics that the motion of the fluid is described. On the other hand, in the setting of the group $G_{\mathcal A}$ of invertible elements of a $C^*$-algebra $\mathcal A$, it is more natural to consider the Finsler metric obtained by left-translating the spectral norm of the algebra, and the linear connections that arise in that setting are not necessarilly the Levi-Civita connections of a Riemannian metric;  this is the viewpoint adopted by Corach, Porta and Recht (see \cite{corach} and the references therein). By means of the action $(g,a)\mapsto gag^*$ of $G_{\mathcal A}$ on the positive cone $\Omega$ of the algebra $\mathcal A$, it is possible to study the relation among the geometries of the group $G_{\mathcal A}$, the stabilizer of the action $U_{\mathcal A}$ (the unitary group of $\mathcal A$), and the quotient space $\Omega\simeq G_{\mathcal A}/U_{\mathcal A}$, as shown in  \cite{corach2}. The metric in the cone obtained by this construction matches Thompson's part metric described by Nussbaum in \cite{nussbaum}. This viewpoint of homogeneous spaces was latter transported to several groups of operators and their homogeneous spaces, see \cite{alrv,conde,maes,larhs} to mention a few.  Our intention with this paper is to geometrize the structure group $\Str$ of a Jordan Banach algebra $\V$, following the guidelines of the previous remarks. In particular, the subgroup $\GO\subset\Str$ preserving the positive cone $\Omega\subset\V$ fulfills the role of the group $G_{\mathcal A}$, by means of the action $(g,a)\mapsto g(a)$ on the cone $\Omega$, and the group of automorphisms of the algebra $\Aut$ takes the place of the unitary group in the previous setting. With the recent developments on the topology of these groups obtained in \cite{larluna1}, we ensure that the rectifiable distances induced by our Finsler metrics in these infinite dimensional groups give in fact their manifold topology, which coincides with the norm topology of $\bv$. 

\smallskip

We now describe the organization and the main results of this paper: in Section \ref{JB} we review the main objects of the paper, which are Jordan Banach algebras $\V$, their positive cone $\Omega$, and the functional calculus in $\V$. Then we describe the main Banach-Lie groups acting on $\V$, which are all embedded subgroups of $\glv$: the group of automorphisms of the algebra $\Aut$, the structure group $\Str$ and the subgroup $\GO$ of $\Str$ preserving the positive cone $\Omega$. The Banach-Lie algebra of the structure group is
\[
\str \simeq \mathbb L\oplus \der
\]
where $\mathbb L$ are the left multiplication operators $L_xy=x\circ y$, $x,y\in \V$ and $\der$ are the derivations of the Jordan algebra $\V$; the Lie algebra $\str$ has an involution $V^{\dagger}=(L_x+D)^{\dagger}=L_x-D$. 

In Section \ref{spr} we present invariant connections in the group $\Str$, and the main results of this section follow:

\textit{Let $\gamma\subset \Str$ be a smooth path, let $\mu\subset \bv$ be a vector field along $\gamma$ in $\Str$, then if we denote $X=\gamma^{-1}\gamma',Y=\gamma^{-1}\mu\in \str$, then 
\[
D_t\mu = \mu'- \frac{1}{2}\gamma(XY+YX+X^{\dagger}Y+Y^{\dagger}X- XY^{\dagger}-YX^{\dagger})
\]
is a covariant derivative without torsion in $\Str$. If $V=L+D\in \str$ and $g\in Str$, the exponential map of this connection is $\Exp_g(gV)=ge^{t(L-D)}e^{2t D}$, and along geodesics parallel transport can be explicitly computed (Proposition \ref{parat}). If $\V$ is a finite dimensional  Euclidian Jordan algebra,  this is the Levi-Civita connection of the left-invariant Riemannian metric in $\Str$
$$
\langle V,W\rangle_g=1/2\Tr(g^{-1}V (g^{-1}W)^{\dagger}+g^{-1}W (g^{-1}V)^{\dagger}), \quad g\in\Str,\;V,W\in T_g\Str.
$$
}
We compute the curvature of this connection and we show that it reduces to the subgroups $\GO$ and $\Aut$, and we give a full description of the Killing fields.  Then in Section  \ref{conesym} we present the cone $\Omega$ as a Cartan homogeneous space of the group $\GO$, therefore carrying a natural connection (the symmetric connection) for which the geodesics are those of the Thompson's part metric \cite{nussbaum}.

In Section \ref{finsleromega} we deal with the Finsler geometry of the positive cone $\Omega$. As a byproduct of the Shirshov-Cohn Theorem we derive a formula linking the differential of the exponential map of the connection described previously for $\Omega$, and a linear operator acting on a $C^*$-algebra (Lemma \ref{rep1}). With this identity, we  prove the minimality of the geodesics of the Thompson part-metric for any symmetric gauge norm in $\V$ (Theorem \ref{minig}), extending significantly the minimility property of these paths, which was only known so far for the Jordan algebra norm (see Nussbaum \cite{nussbaum} and Neeb \cite{neeb}).

In Section \ref{fgo} we endow $\Str$ and $\GO$ with a left-invariant metric using the uniform norm in the Lie algebras of the groups, and we show that the one-parameter groups of derivations in $\Aut$ -geodesics of the connection discussed in Section \ref{spr}-  are of minimal length for this Finsler metric (Theorem \ref{mini});  we characterize all such minimizing paths (there are many since the norm is not smooth). Then we prove that the metric in $\Omega$ discussed in the previous section, is in fact the quotient metric of the one given for $\GO$ (Theorem \ref{mismad}). Finally, using the notion of horizontal lift of geodesics, we show that one-parameter groups in $\GO$, with initial speed $L_v$ (where $L_vw=v\circ w$ as before), are indeed minimal for the distance induced by our left-invariant Finsler metric in $\GO$ (Corollary \ref{coroh}), relating explicitly the geometry in $\GO$ with the geometry in $\Omega$.

\section{Jordan Banach algebras and the structure group}\label{JB}

Let $\V$ be a real vector space with product $\circ$, possibly infinite dimensional. Then $(\V,\circ)$ is a \textit{Jordan algebra} if $\circ$ is commutative and
\begin{equation}\label{eq:jordanAlg}
	x^2 \circ (x \circ y) = x\circ (x^2 \circ y).
\end{equation}

Every associative algebra can be made into a Jordan algebra with the Jordan product $x \circ y =1/2(xy + yx)$. We will denote as $\bv$ the associative algebra of operators of linear bounded operators on $\V$, and $\glv$ will denote the invertible elements of $\bv$.

\begin{defi}[Quadratic representation]For fixed $x\in \V$, define the operator $L_x: \V \to \V$ by means of $L_xy = x\circ y$, and consider the linear operator $U_x = 2L_x^2 - L_{x^2}$. This is the \textit{quadratic representation} of of $\V$ in $\bv$. If we compute the quadratic representation in an associative algebra with  Jordan product as defined above, we obtain $U_xy = xyx$. The quadratic representation gives us a bilinear representation:
$$
	U_{x,y} = 1/2(U_{x+y} - U_x - U_y) = \frac{1}{2} D_yU(x) = L_xL_y + L_yL_x - L_{x\circ y}=U_{y,x}.
$$
where $D_y f(x)$ denotes the differential of the map $f$ at the point $y$, in the direction of $x$. In an associative algebra with the Jordan product defined above, $U_{x,y}z = \frac{1}{2}(xzy+yzx)$. We define the bilinear \textit{$V$-operators} as a permutation of the $U$, that is $V_{x,y}(z) = U_{x,z}(y)$. The most important property of this representation is the \textit{fundamental formula}:
\begin{equation}\label{eq:fundamf}
	U_{U_xy} = U_xU_yU_x.
\end{equation}
\end{defi}

\begin{defi}[Invertible elements and spectrum]\label{invert}
	An element $x\in \V$ is \textit{invertible} if there exists an element $y\in \V$ such that $U_xy = x$ and $U_xy^2 = 1$. An element $x$ is invertible in $\V$ if and only if $U_x$ is an invertible operator (for a proof see \cite[Part II, Criterion 6.1.2]{mccrimmon}). Moreover, $x,y$ are invertible if and only if $U_xy$ is also invertible: this is plain from  $U_{U_xy}=U_xU_yU_x$ (the fundamental formula) and the previous result. For $x\in \V$, the \textit{spectrum} of $x$ is $\sigma(x) = \{\lambda \in \R \text{ such that } x-\lambda 1 \text{ is not invertible}\}$. 
\end{defi}

\begin{defi}[Positive cone $\Omega$ of a Jordan algebra]	Given $\V$ a Jordan algebra and $x\in \V$, we say $x$ is positive if $\sigma(x)$ is contained in the positive numbers. We denote by $\Omega$ the set of positive elements in $\V$. 	The set $\Omega$ is a convex proper cone: $\Omega \cap (-\Omega)=\emptyset$, thus there is a partial order in $\V$:  given $x,y\in \V$, we say that $x<y$ if $y-x\in \Omega$. The closure  $\overline{\Omega}$ is the set of elements with nonnegative spectrum, see Remark 2.28 in \cite{larluna1}.
\end{defi}

We now add more structure: we will provide a norm for $\V$ related to the order structure. The general reference for JB-algebras is the book \cite{stormer}.
\begin{defi}
	Let $\V$ be a Jordan algebra with norm $\|\cdot\|$ such that $(\V,\|\cdot\|)$ is a Banach space. We say $\V$ is a JB-algebra if
$$ 1)\quad \|x\circ y\| \leq \|x\| \,\|y\| \qquad 2) \quad \|x^2\| = \|x\|^2 \qquad 3)\quad \|x^2\| \leq \|x^2 + y^2\|.
$$
\end{defi}

Then  every JB-algebra is \textit{totally real}: if $x^2 + y^2 = 0$, then both $x=y=0$.  The \textit{order norm} in $\V$ is the norm induced by the partial order and the chosen unit order $e=1$ in the cone $\Omega\subset \V$:
	\begin{equation*}
		\|x\| = \inf\{\lambda > 0 : -\lambda 1 < x < \lambda 1\}.
	\end{equation*}

\begin{rem}[Square roots and logarithms]\label{expsq}
Every JB-algebra is archimedean: for every $x$ in $\V$ exists $\lambda > 0$ such that $-\lambda 1< x < \lambda 1$ (see \cite[Proposition 3.3.10]{stormer} for a proof). Moreover, the order norm coincides with the original norm of the space $\V$ \cite[Proposition 3.3.10]{stormer}. The order norm on a JB-algebra allows us to do continuous functional calculus on $\V$. Let $\mathcal{C}(x)$ be the associative commutative Banach algebra in $V$ generated by $x$; it is isometrically isomorphic to $C(\sigma(x))$. A proof of this can be found in \cite[Theorem 3.2.4]{stormer}, and it uses a complexification of $\mathcal{C}(x)$ and the known spectral theorem for complex algebras. Then we can characterize $\Omega$ as the cone of squares of $\V$: every positive element has a square root, and as $\sigma(x^2) = \{\lambda^2, \lambda \in \sigma(x)\}$, every square is positive. Likewise, we can also characterize $\Omega=e^{\V}$: every positive element has a real logarithm, and  $\sigma(e^x) = \{e^\lambda, \lambda \in \sigma(x)\}$, so the exponential of every element is positive.
\end{rem}

We will consider a group of automorphisms that will act on $\Omega$; we want it to be a Banach-Lie group. To give an appropiate Lie and manifold structure to the group of transformations that fix the cone, we will study first a larger group and derive the structure from there. 
\begin{defi}[Structure Group]
	The structure group of $\V$ is the set of $g\in \glv$ such that there exists another $g^*\in \glv$ with
	\begin{equation*}
		U_{gx} = g U_x g^*
	\end{equation*}
	for all $x$ in $\V$.
	We denote the structure group as $\Str$; it is a closed subgroup of $\glv$.
\end{defi}

\begin{rem}[The adjoint and group operations]
	If $g$ and $h$ belong to $\Str$, then $gh$ belongs to $\Str$ and $(gh)^* = h^*g^*$, and $g^{-1}$ belongs to $\Str$ with $(g^{-1})^* = (g^*)^{-1}$. If $g\in \Str$ so does $g^*$, and $(g^*)^* = g$. Let $y$ be an invertible element in $\V$, then $U_y$ belongs to the structure group and $U_y^* = U_y$. This is a direct consequence of the fundamental formula and from the fact that $y$ is invertible if and only if $U_y$ is invertible.
\end{rem}


\begin{defi}We say that $k\in \glv$ is a \textit{multiplicative automorphism} of $\V$ (or an automorphism of $\V$ for short) if $k(1)=1$ and  $k(a\circ b)=k(a)\circ k(b)$ for all $a,b\in \V$. We will denote this set with $\Aut$; it is a closed subgroup of $\Str$. 
\end{defi}

Note that $k(L_xy)=k(x\circ y)=(kx)\circ (ky)=L_{kx}ky$, then $kL_vk^{-1}=L_{kv}$  and  $U_{kv}=kU_vk^{-1}$ for all  $v\in \V$. In particular $\Aut$ preserves the invertibles and $k(v)^{-1}=k(v^{-1})$. Moreover $\Aut$ preserves the cone $\Omega$, since $k(v^2)=(kv)^2$ and the cone can be characterized as the set of invertible squares in $\V$ (Remark \ref{expsq}). 

\medskip

Let $\der$ be the subspace of \textit{derivations} i.e. $D\in \bv$ such that $D(x\circ y)=Dx\circ y+x\circ Dy$ for all $x,y\in \V$. It is plain that $\der\subset \bv$ is a Banach-Lie subalgebra. The following was proved in \cite{larluna1}:

\begin{teo}  $\Str\subset \glv$ is an embedded Banach-Lie group with  
$$
\lie(\Str)=\str = \{H \in \bv: 2U_{x,Hx} = HU_x - U_x\overline{H}\quad \forall\, x\in \V \},
$$
where  $\overline{H}=H-2U_{H1,1}$. Also $\Aut\subset \Str$ is an embedded Banach-Lie subgroup, with Banach-Lie algebra $\lie(\Aut)=\der\subset \str$.
\end{teo}

\smallskip

It is well-know that for any $x\in \V$, then  $e^{2L_x} = U_{e^x}$. As for every $t$ and every $v\in \V$ we have that $e^{tL_v}\in  \Str$, the left-multiplication operators $L_v$ belong to $\str$, this is a Banach subspace of $\bv$, which we denote  $\mathbb{L} = \{L_v, v \in \V\}$, i.e.  $\mathbb L\subset \str$.

On the other hand $\der=\{D\in \str: D1=0\}$, and  $\str=\der \oplus \mathbb{L}$. We have that $\der$ is a reductive subalgebra in the following sense:
\begin{equation}\label{cartan}
[\mathbb L,\mathbb L]\subset \der \quad [\mathbb L,\der]\subset \mathbb L\quad [\der,\der]\subset \der,
\end{equation}
i.e. we obtain a \textit{Lie-triple system}.

\begin{defi}
Let $\GO$ be the group of automorphisms that preserve the cone
$$
\GO=\{g\in \glv: g(\Omega)=\Omega\}.
$$
\end{defi}

\begin{rem}\label{goemb}
The inner structure group $\textrm{InnStr}(\V)$ is the subgroup of $\Str$ generated by the quadratic operators $U_x$ with $x\in V$ invertible. Then $\textrm{InnStr}(\V)$ and $\Aut$ are embedded Banach-Lie subgroups of $\Str$, and also $\GO\subset \Str$ (see \cite{larluna1}). Although $\GO$ is contained in the structure group, it does not automatically inherit a differentiable structure. However, if we denote $\Str_0$ the connected component of the identity of $\Str$ (which is open since $\Str$ is a Banach-Lie group), then $\Str_0\subset \GO\subset \Str$ and each inclusion is open and closed, thus $\GO$ is an embedded Banach-Lie subgroup of $\glv$. This is also proved in \cite{larluna1}.
\end{rem}

\section{Sprays and connections in $\GO$ and $\Omega$}\label{spr}

For a smooth map $f:M\to N$ we denote $f_*:TM\to TN$  the fiber bundle map, locally given by $f_{*p}v=Df_p(v)$, where the latter is the usual differential of $f$ at $p$ in the direction of $v$. If $\rho:M\to M$, we will use  $D^2_p\rho$ to denote the second differential of $\rho$ in a local chart.

\bigskip

In what follows we give a short presentation of covariant derivatives without torsion, from the point of view of quadratic sprays introduced by Ambrose, Palais and Singer:

\begin{rem}[Quadratic sprays and connections]\label{quadrspr} Let $M$ be a smooth manifold (at least $C^4$) over the Banach space $X$. For $s\in \mathbb R$, let $\mu_s:TM\to TM$ be multiplication by $s$ in each fiber (likewise for the fiber bundle  $TTM$). A \textit{quadratic spray} in $M$ is a smooth section $F:TM\to TTM$ (i.e. if $\pi:TM\to M$ is the canonical projection, then $\pi_*F=id_{TM}$) such that $F(\mu_s V)=(\mu_s)_* \mu_s F(V)$ for each $V\in TM$. Note that the projection condition can be restated as $F(V)\in T_VTM$ for each $V\in TM$. Locally, a spray is given by a smooth map $F:U\times X\to X$ that we denote $(p,v)\mapsto F_p(V)$. The \textit{Christoffel bilinear operator of the spray} is obtained by polarizing $F$, that is
$$
\Gamma_p(v,w)=\frac{1}{2}(F_p(v+w)-F_p(v)-F_p(w)).
$$
If $X,Y\in\mathfrak X(M)$ are smooth vector fields, any quadratic spray induces a unique covariant derivative without torsion, which locally is given by 
$$
(\nabla_X Y)_p=DY_p(X_p)-\Gamma_p(X_p,Y_p).
$$
In particular the paralell transport $\mu_t=P_t^{t+s}(\gamma)$ for $\nabla$ along a path $\gamma\subset M$ obeys the differential equation $\mu_t'=\Gamma_{\gamma_t}(\gamma_t',\mu_t)$, and Euler's equation for the $\nabla$ geodesics is $\gamma_t''=F_{\gamma_t}(\gamma_t')$. This is discussed with further detail in \cite[Chapter IV]{lang} or \cite[Capitulo 6]{larestr}).
\end{rem}

\medskip

In what follows, $M$ is Banach smooth manifold, modelled on a Banach space $X$.

\begin{rem}\label{frellie2}
If $f$ is smooth and injective, $f:U\to V$ with  $U,V$ open in $M$, and  $X\in \mathfrak X(U)$ consider  the push-forward $f_*X\in \mathfrak(V)$ given locally by $f_*X(f(p))=Df_p(X_p)$. 
\end{rem}

\begin{defi}[Automorphisms of the connection, affine transformations]
 More generally, let $M,\overline{M}$ be smooth manifolds with spray, let $\nabla,\overline{\nabla}$ be their respective connections. Let $f:M\to\overline{M}$ be a local diffeomorphism, we say that $f$ is an \textit{affine transformation} if for each $U\subset M$ where $f|_U:U\to f(U)$ is bijective, and every pair of vector fields $X,Y$ in $U$, we have
$$
\overline{\nabla}_{f_*X}f_*Y=f_* (\nabla_XY).
$$
In particular if $U,V\subset (M,\nabla)$ and $f:U\to V$ is affine, we say that $f$ is a  \textit{connection automorphism}.  These automorphisms will be denoted by $\texttt{Aut}(M,\nabla)$.
\end{defi}

\begin{rem}[Invariance of parallel transport by automorphisms]\label{isotrans}
If $f\in \texttt{Aut}(M,\nabla)$, then $f$ maps geodesics into geodesics. Morevoer, if $\gamma\subset M$, $p=\gamma(a)$, and $\mu(t)=P(\gamma)_a^t v$ indicates parallell transport of $v\in T_pD$ along $\gamma$, then $\eta=Df_{\gamma}\mu$ is parallell transport of $w=Df_pv$ along $\beta=f\circ\gamma$. 
\end{rem}



Let $X\in \mathfrak X(M)$, we denote  $\rho(t,p)=\rho_t(p)$ the local flow of $X$, that verifies $\rho'_t(p)=X(\rho_t(p))$, $\rho_0(p)=p$. Now we introduce the \textit{connection Killing fields}, see \cite{lang} for proofs of the following:


\begin{teo}[Killing fields]\label{killspray} Let $X\in\mathfrak X(M)$, let $\rho_t$ be the local flow of $X$. The following are equivalent
\begin{enumerate}
\item $\rho_t\in \texttt{Aut}(M,\nabla)$ for all $t$.
\item For any $t$ and any $V\in TM$, we have $F((\rho_t)_*V)= (\rho_t)_{**} F(V)$
\item For any geodesic $\gamma\subset M$, the field $X$ is Jacobi along $\gamma$.
\item For any $Y,Z\in \mathfrak X(M)$ we have $[X,\nabla_YZ]=\nabla_{[X,Y]}Z+\nabla_Y [X,Z]$.
\end{enumerate}
A field $X$ verifying any of these conditions is a \textit{Killing field} of $(M,\nabla)$, we denote with $Kill(M,\nabla)$ to the set of Killing fields; it is a Lie subalgebra of $\mathfrak X(M)$.
\end{teo}

Note that by the second property above, Killing fields can be described by the connection $\nabla$  or by the quadratic spray $F$ inducing the connection; we will use $Kill(M)$ to denote the Killing fields, when the spray or connection is clear from the context. Automorphisms and Killing fields are determined up to first order, see \cite{neeb} which extends previous results of \cite[Ch. XIII]{lang}.

\subsection{Invariant connections in $\Str$ and $\GO$}

Since $\glv$ is open in the Banach space $\bv$, we will identify tangent spaces of subgroups $G\subset \glv$ with the left translation of their Lie algebras, that is if $g\in G\subset \glv$ is a Lie group, then $T_gG=gT_1G=g\lie(G)$. 

Then, a typical path in $TG$ is of the form $\gamma_t=g_tv_t$, where $g_t$ is a path in $G$, and $v_t$ is a path in $\lie(G)$. Thus a typical element of $TTG$ can be identified with the speed of $\gamma$, that is $\gamma'=g'v+gv'$. But since $g_t\subset G$, then $g'\in T_gG$ and it must be of the form $g'=gw$ for some $w\in \lie(G)$, and also $v'=z\in \lie(G)$. Thus \textit{a typical element of $T_VTG$ for $V=gv\in TgG$ must be of the form}
$$
Z=gwv+gz=g(wv+z),\qquad g\in G, \quad v,w,z\in \lie(G).
$$
Therefore a left-invariant connection spray in $\Str$ must be of the form
\begin{equation}\label{Fgeneral}
F_g(gV)=g(V^2+B(V,V)),\qquad g\in \Str,\, V\in \str.
\end{equation}
for some continuous symmetric bilinear operator $B:\str\times\str\to\str$, and it is apparent that each quadratic operator like this defines a left-invariant spray in $\Str$.

\medskip

\begin{defi}[The $\dagger$ operation] In $\str=\mathbb L\oplus\der$ we have the linear operation $(L_x+D)^{\dagger}=L_x-D$, where $L_x$ is multiplication by $x\in\V$ and $D$ is a derivation; the opposite of ${\dagger}$ is a Lie algebra homomorphism (see Lemma \ref{overline} below). 
\end{defi}

\begin{defi}[Left-invariant spay]\label{Fstr} Consider the case of the spray
\begin{equation}\label{sprayF}
F_g(gV)=g(V^2+[V^{\dagger},V])=g(V^2+V^{\dagger}V-VV^{\dagger}),\qquad g\in \Str,\, V\in \str,
\end{equation}
that is $g^{-1}F_g(g(L_x+D))=D^2+(L_x)^2+3L_xD-DL_x=(L_x+D)^2-2L_{Dx}$. If $x,y\in \V$ and $d,D\in \der$, then the Christoffel bilinear operator $g^{-1}\Gamma_g(g(L_x+D),g(L_y+d))$ equals to
\begin{align}\label{gamaLD}
  &1/2 (  L_yL_x+L_xL_y  +3L_xd+3L_yD -DL_y-dL_x     +dD +Dd   ).
\end{align}
\end{defi}

\medskip

The decomposition of $\str$ induces a decomposition of $T\Str$ in the direct sum of two vector bundles, the \textit{horizontal bundle} and the \textit{vertical bundle}, where for each $g\in\Str$ 
$$
\mathcal H_g=\{gL_x: x\in \V\}=g\mathbb L\qquad \textrm{ and }\qquad \mathcal V_g=\{gD: D\in \der\}=g\der.
$$
For horizontal vectors, we have $F_g(gL_x)=g(L_x)^2$ while for vertical vectors it is $F_g(gD)=gD^2$. On the other hand, it is easy to check that this will be the case for a general spray $F$ in $\Str$ of the form (\ref{Fgeneral}) if and only if $B(L_x,L_x)=B(D,D)=0$ for all $x,D$.

\begin{rem}[Geodesics]\label{geoGO} It is easy to check that the unique geodesic of the spray with $\gamma(0)=g$ and $\gamma_0'=V=L_x+D$ is 
$$
\gamma(t)=ge^{t(L_x-D)}e^{2tD}.
$$
\end{rem}

\begin{rem}[$\GO$ and $\Aut$ are totally geodesic in $\Str$]\label{totageo} It is apparent that this spray can be restricted to the open subgroup $\GO$, and also to the Banach-Lie subgroup $\Aut$. Since $e^{t(L_x-D)}\subset \Str_0\subset \GO$, and the automorphisms $e^{2tD}$ also preserve the positive cone $\Omega$,  it is clear that a geodesic with initial position $g\in \GO$, stays inside $\GO$ for all $t$. Thus $\GO$ is totally geodesic in $\Str$.

\smallskip

On the other hand, for a geodesic to have initial speed in $T\Aut$, it is necessary and sufficient that $g=\gamma(0)\in \Aut$ and that $L_x=0$, thus the geodesic must be of the form $\gamma(t)=ke^{tD}$ and again we wee that $\Aut$ is totally geodesic in $\Str$. In particular geodesics of $\Aut$ are (left translations of) one-parameter groups.
\end{rem}

\begin{prop}[Paralell transport]\label{parat} Let $\gamma\subset\Str$ be a smooth path, let $v_t=\gamma^{-1}\gamma'=L_{y_t}+d_t$ with $d_t$ a path in $\der$ and $y_t$ a path in $\V$. Then $\mu=\gamma(L_{x_t}+D_t)$ is paralell transport along $\gamma$ if and only if
\begin{align*}
L_x'&=-3/2[d,L_x]+1/2[L_y,D] \quad \textrm{ inside }\mathbb L\\
D'&=-1/2[L_y,L_x]-1/2[d,D] \quad \textrm{ inside } \der.
\end{align*}
\end{prop}
In particular if $\gamma=ge^{t(L_{y_0}-d_0)}e^{2td_0}$ is a geodesic then $v_t=e^{-2t\ad d_0}(L_{y_0}+d_0)$ and we can solve explicitly 
$$
\mu_t=\gamma_t e^{-2t\ad d_0} e^{tM}\gamma_0^{-1}\mu_0=ge^{t(L_{y_0}-d_0)} e^{tM}g^{-1}\mu_0,
$$
where
$$
M=1/2\left(\begin{array}{cr}\ad d_0 & \ad L_{y_0} \\ -\ad L_{y_0} & 3\ad d_0 \end{array}\right).
$$
\begin{proof}Let $\gamma$ be any smooth path in $\Str$, let  $\mu_t=\gamma(L_{x_t}+D_t)$ for some smooth map $x_t$ in $V$ and some smooth map $D_t$ in $\der$. We first compute 
\begin{align*}
\gamma^{-1}\mu'& =\gamma^{-1}\gamma'(L_x+D)+L_x'+D'=(L_y+d)(L_x+D)+L_x'+D'\\
&=L_yL_x+L_yD+dL_x+dD+L_x'+D'.
\end{align*}
Comparing this with (\ref{gamaLD}) we arrive to
$$
L_x'+D'=1/2[L_x,L_y]+3/2[L_x,d]+1/2[L_y,D]+1/2[D,d].
$$
Recalling the relations (\ref{cartan}), it must be $L_x'=3/2[L_x,d]+1/2[L_y,D] $ and $D'=1/2[L_x,L_y]+1/2[D,d]$ as claimed in the first formulas.  Now assumme $\gamma$ is a geodesic as stated, we let 
$$
\epsilon=L_z+\widetilde{D}=e^{2t\ad d_0}\gamma^{-1} \mu=e^{2t\ad d_0}L_x+e^{2t\ad d_0}D.
$$
Since $L_z'+\widetilde{D}'=\epsilon'=e^{2t\ad d_0}( 2[d_0,L_x]+2[d_0,D]+L_x'+D')$, after pluggin this in the equations for $L_x,D$, we obtain that $L_z'=1/2[d_0,L_z]+1/2[L_{y_0},\widetilde{D}]$ and also that $\widetilde{D}'=-1/2[L_{y_0},L_z]-3/2[d_0,\widetilde{D}]$. Hence 
$$
\left(\begin{array}{c}L_z' \\ \widetilde{D}'\end{array}  \right) =1/2\left(\begin{array}{cr}\ad d_0 & \ad L_{y_0} \\ -\ad L_{y_0} & 3\ad d_0 \end{array}\right)\left(\begin{array}{c}L_z \\ \widetilde{D}\end{array}  \right).
$$
Hence it must be $\epsilon=L_z+\widetilde{D}=e^{tM}\epsilon_0$, and what's left is apparent.
\end{proof}

\subsubsection{Motivation: the Riemannian metric of a  Euclidean algebra}

In this section we give some background and motivation for the choice of this particular affine connection in  $\Str$. Let $\V$ be a finite dimensional Euclidean Jordan algebra, then if we denote $\Tr$ the trace of $\bv$, the trace in $\V$ given by $\tr(x)=\Tr(L_x)$ induces a positive definite inner product on $\V$ by means of $(v|w)=\tr(vw)$.  This trace in $\V$ is invariant for automorphisms since
$$
\tr(kx)=\Tr(L_{kx})=\Tr(kL_xk^{-1})=\Tr(L_x)=\tr(x).
$$
In particular differentiating $\tr(e^{tD}x)=\tr(x)$ we see that $\Tr(L_{DX})=\tr(Dx)=0$ for any $x\in\V$ and any $D\in \der$. 

\begin{teo}If $\V$ is a (finite dimensional) Euclidean Jordan algebra with inner product $(v|w)=\tr(vw)$, then $\Str$ admits the left invariant Riemannian metric
$$
\langle V,W\rangle_g=1/2\Tr(g^{-1}V (g^{-1}W)^{\dagger}+g^{-1}W (g^{-1}V)^{\dagger}), \quad g\in\Str,\;V,W\in T_g\Str
$$
or equivalently
$$
\langle g(L_x+D),g(L_y+\widetilde{D})\rangle_g=\Tr(L_xL_y)-\Tr(D\widetilde{D}),
$$
The connection $\nabla$ induced by the spray (\ref{sprayF}) is the Levi-Civita connection of this metric.
\end{teo}
\begin{proof}
It is clear that the formula defines a left-invariant bilinear form, we need only to check that $\langle V,V\rangle_g>0$ if $V\ne 0$. If $V=L_x+D\in \str$, then $\langle V,V\rangle_1=\Tr( (L_x)^2)-\Tr(D^2)$. Now  $Dx^2=2xDx$ thus $D^2x^2=2xD^2x+2 (Dx)^2$, hence $-xD^2x=1/2(Dx)^2-1/2D^2x^2$ and then $\tr(-xD^2x)=1/2\tr((Dx)^2)\ge 0$ (since $D^2x^2=Dv$ for $v=Dx^2$) and it is $0$ only if $Dx=0$. Let $B=\{v_1,\dots,v_n\}$ is a basis of $\V$, and compute
$$
-\Tr(D^2)=\sum_i (-D^2v_i|v_i)=\sum_i \tr(-v_i D^2v_i)=\sum_i \tr( (Dv_i)^2)\ge 0.
$$
Take a complete system of orthogonal primitive idempotents $c_i$ (which form an orthogonal basis) such that $x =\sum_i \lambda_i c_i$ for some real $\lambda_i$ (see the second spectral theorem in \cite[Theorem III.1.2]{faraut}). Then
$$
\Tr((L_x)^2)=\sum_i ((L_x)^2c_i|c_i)=\sum_i \lambda_i^2\tr(c_i^2)\ge 0.
$$
Thus our bilinear form is non-negative. But if $\langle V,V\rangle_1=0$, then it must be $-\Tr(D^2)=0$ (which implies that $Dv_i=0$ for all $i$, thus $D=0$), and also that $\Tr(L_x^2)=\sum_i \lambda_i^2\tr(c_i^2)=0$, which also implies that $x=0$. Then $\langle V,V\rangle_g=0$ only if $V=0$. 

Now we prove that the connection (equivalently, the spray $F$) is the metric connection for this riemannian metric. Since $\nabla$ is a connection without torsion, it suffices to check that it is compatible with the metric. Let $\mu,\eta$ be vector fields along $\gamma\subset\Str$,  we write $v=\gamma^{-1}\gamma'$,  $\tilde{\mu}=\gamma^{-1}\mu$. Then by  Remark \ref{quadrspr}
$$
\gamma^{-1}D_t\mu=\gamma^{-1}\mu'-\gamma^{-1}\Gamma_{\gamma}(\gamma',\mu)=\gamma^{-1}\mu'-1/2(\tilde{\mu}v+v\tilde{\mu}+[\tilde{\mu}^*,v]+[v^*,\tilde{\mu}]).
$$
We have a similar expression for $\gamma^{-1}D_t\eta$. Thus using the cyclicity of the trace, after a tedious but straightforward computation we obtain
\begin{equation}\label{1}
\langle D_t\mu,\eta\rangle_{\gamma}+\langle D_t\eta,\mu\rangle_{\gamma} =  \Tr(\tilde{\eta}^{\dagger}\gamma^{-1}\mu')+\Tr(\tilde{\mu}^{\dagger}\gamma^{-1}\eta')-\Tr(v\tilde{\mu}\tilde{\eta}^{\dagger})-\Tr(v\tilde{\eta}\tilde{\mu}^{\dagger}).
\end{equation}
Note that $\tilde{\mu}'=(\gamma^{-1}\mu)'=-\gamma^{-1}\gamma'\gamma^{-1}\mu+\gamma^{-1}\mu'=-v\tilde{\mu}+\gamma^{-1}\mu'$, and similarly for $\tilde{\eta}'$. Thus the last term in (\ref{1}) is equal to
$$
\frac{d}{dt}\Tr(\tilde{\mu}\tilde{\eta}^{\dagger})=\frac{d}{dt}\langle \mu,\eta\rangle_{\gamma},
$$
and we are done.
\end{proof}

The example above can be also presented in the infinite dimensional setting of Jordan-Hilbert algebras with a finite trace, such as the special Jordan algebra of (self-adjoint) Hilbert-Schmidt operators acting on a complex Hilbert space, see  \cite{alrv} for further discussion.

\bigskip

We now introduce symmetric spaces according to Loos  \cite{loos}, to be able to deal with the geometry of the cone $\Omega$; see also Neeb's paper \cite{neeb} for the Banach setting:

\begin{defi}[Symmetric space]\label{esime} 
Let $M$ be a Banach manifold $\mu:M\times M \to M$ smooth, denote $\mu(x,y)=x\cdot y=S_x(y)$. We say that $(M,\mu)$ is a \textit{symmetric space} if the following axioms hold for $x,y,z\in M$:
\begin{align*}
&(S1)\;x\cdot x=x\qquad (S2)\; x\cdot (x\cdot y)=y\qquad (S3)\;x\cdot( y\cdot z)=(x\cdot y)\cdot (x\cdot z)\\
&(S4)\; \textrm{ Each }x\in M  \textrm{ has a neighbourhood } U \textrm{ s.t. for } y\in U,  \textrm{ if }x\cdot y=y \textrm{ then }x=y.
\end{align*}
\end{defi}

For each $p\in M$, the map $S_p:M\to M$ is a diffeomorphism by $(S2)$, because this is equivalent to $S^2_p=id_M$. The first axiom tells us that $p$ is a fixed point of $S_p$ and the fourth axiom tells us that  $p$ is an isolated fixed point of $S_p$. These maps $S_p$ are known as \textit{symmetries} of the symmetric space $(M,\mu)$.

\begin{defi}[Automorphisms of a symmetric space] Let $(M,\mu)$ be a symmetric space. We say that a local diffeomorphism  $f:M\to M$ is an \textit{automorphism} of $(M,\mu)$ if $f(\mu(x,y))=\mu(f(x),f(y))$ for all $x,y\in M$. We denote this group by  $Aut(M,\mu)$. It is immediate from  $(S3)$ that the symmetries $S_p$ ($p\in M$) belong to $Aut(M,\mu)$.
\end{defi}

\begin{rem}[$TM$ as a symmetric space]\label{teme} Let $(M,\mu)$ be a symmetric space, then for all $p\in M$ we have  $(S_p)_{*p}=-id_{T_pM}$ because of $(S4)$. For $V,W\in TM$ the product $V\cdot W=\mu_*(V,W)$ defines a symmetric space structure in  $TM$. If $v,w\in T_pM$ then $v\cdot w=2v-w$. Moreover if $f\in Aut(M,\mu)$ then  $f_*\in Aut(M,\mu_*)$. For $V=(p,v)\in TM$ let $\Sigma_V:TM\to TM$ be given by $\Sigma_V(W)=\mu_*(V,W)$, more precisely
$$
\Sigma_{(p,v)}(q,w)= \mu_{*(p,q)}(v,w).
$$
We will denote with  $Z:M\to TM$ the zero section,  $Z(q)=(q,0)$ for  $q\in M$.
\end{rem}

\begin{teo}[Connection of a symmetric space]\label{geosim}
Let  $(M,\mu)$ be a connected symmetric space. If $V=(p,v)\in TM$, $\Sigma_V\circ Z:M\to TM$ and $(\Sigma_V\circ Z)_*:TM\to TTM$, then
\begin{enumerate}
\item $F(V)=-(\Sigma_{V/2}\circ Z)_*V$ is a quadratic spray in $M$.
\item $Aut(M,F)=Aut(M,\mu)$.
\item $F$ is the unique quadratic spray in $M$ which is invariant for all the symmetries $S_p$. In fact locally we have $F_p(v)=\frac{1}{2} (D^2S_p)_p(v,v)$ for all $v\in T_pM$.
\item $(M,F)$ is geodesically complete.
\item If $R$ is the curvature tensor of $(M,F)$ then $\nabla R=0$.
\end{enumerate}
\end{teo}
\begin{proof}
See \cite{neeb} and for further details see \cite[Section 7.5]{larestr}.
\end{proof}


\begin{rem}\label{curvatura2}If $F$ is a quadratic spray, the bilinear Christoffel operator is obtained by polarization, i.e. $2\Gamma(V,W)=F(V+W)-F(V)-F(W)$. However for symmetric spaces we have a simpler expresssion (see \cite[Section 7.5]{larestr}):
$$\Gamma_p(V,W)=-\frac{1}{2}(\Sigma_{V}\circ Z)_{*p}W=-\frac{1}{2}(\Sigma_{W}\circ Z)_{*p}V.$$
If $R$ is the curvature tensor we obtain
$$R(U,V)=\frac{-1}{4}\{(\Sigma_U\circ Z)_*(\Sigma_V\circ Z)_*- (\Sigma_V\circ Z)_*(\Sigma_U\circ Z)_*\}.$$
\end{rem}

\medskip

The viewpoint of homogeneous spaces will be also useful so we present it here:

\begin{defi}[Symmetric groups and Cartan symmetric spaces] If $G$ is a Banach-Lie group, and  $\sigma:G\to G$ is an automorphism, we say that $(G,\sigma)$ is a  \textit{symmetric Lie group} if $\sigma^2=id_G$. Consider 
$$
G^{\sigma}=\{g\in G: \sigma(g)=g\},
$$
for an open subgroup $K\subset G^{\sigma}$, we say that the space  $M=G/K$ is a \textit{Cartan symmetric space}.
\end{defi}

We will denote $q:G\to M$ to the quotient map, let also $o=q(1)$ and $g\cdot o= gK$ denote the quotient elements, let $\pi(g,h\cdot o)=\pi_{h\cdot o}(g)=\ell_g(h\cdot o)=(gh)\cdot o$ be the action of $G$ in $M$. So $\ell_g$ is an automorphism of $M$.

\begin{prop}\label{gcartan} 
If  $\sigma_*:\lie(G)\to \lie(G)$ denotes the differential of $\sigma$ at the identity of $G$, then $\sigma_*^2=1$, thus  $\lie(G)=\mathfrak m\oplus \mathfrak k$ where
$$ 
\mathfrak m=\{x\in \lie(G): \sigma_*x=-x\},\quad \mathfrak k=\{y\in \lie(G): \sigma_*y=y\}.
$$
Moreover (see \cite[pag. 134]{neeb}, \cite[Chapter III$\S$1.6]{bourbaki} or \cite[Section 7.5.2]{larestr} for full details):
\begin{enumerate}
\item $\mathfrak g=\mathfrak m\oplus\mathfrak k$ is a \textit{Cartan decomposition}, that is
$\; [\mathfrak k,\mathfrak k]\subset \mathfrak k,\; [\mathfrak k,\mathfrak m]\subset \mathfrak m,\;  [\mathfrak m,\mathfrak m]\subset \mathfrak k$.
\item $K<G$ is an embedded Banach-Lie subgroup and $\lie(K)=\mathfrak k$.
\item $M$ has a unique smooth structure such that $q$ is a smooth quotient map, $q_{*1}:\mathfrak m\to T_oM$ is an isomorphism and  $\ker q_{*1}=\mathfrak k$. A chart of $M$ around $o$ is given by the inverse of $Exp:\mathfrak m\to M, \, x\mapsto e^x\cdot o$ with a suitable restriction to a neighbouhood of $0\in \mathfrak m$.
\end{enumerate}
\end{prop}

\begin{rem}[The symmetric space structure of the quotient space $M=G/K$]\label{simenquot}
In  $M=G/K$ define
\begin{equation}\label{symstr}
\mu(g\cdot o,h\cdot o)=g\sigma(g)^{-1}\sigma(h)\cdot o.
\end{equation}
This is well-defined and $(M,\mu)$ obeys the axioms of a symmetric space, in particular   $G$ acts on  $M$ by  $\mu$ automorphisms: $\ell_g\mu(p,q)=\mu(\ell_g p,\ell_g q)$. 
\end{rem}

Abusing notation a bit, if $g\in G$ and $x\in \lie(G)$, we will denote by $gx$  the differential of left multiplication by $g$ in the group $G$. See \cite[Theorem 3.6]{neeb} and \cite[Section 7.5]{larestr} for a proof of the following:

\begin{teo}[Geodesics, paralell transport and Killing fields of $(M,F)$]\label{geosimcar} Let $M=G/K$ be connected with $(G,\sigma)$ symmetric Banach-Lie group and  $K\subset G^{\sigma}$ open subgroup. Let $(M,F)$ be the spray induced by the symmetric structure. Then
\begin{enumerate}
\item If $g\in G$ and $v=q_{*g}(gx),w=q_{*g}(gy)\in T_{g\cdot o}M$ with $x,y\in\mathfrak m$, then 
$$
\Gamma_{g\cdot o}(v,w)= -(\ell_g)_{*o}\Gamma_o(q_{*1}x,q_{*1}y)=-\frac{d^2}{ds\,dt}\bigg{|}_{s=t=0}q(ge^{-sy}e^{tx}).
$$
\item Let  $p=g\cdot o\in M$ and $v=q_{*1}x\in T_oM$ with $x\in \mathfrak m$. The unique geodesic $\gamma$ with $\gamma_0=p,\gamma_0'=(\ell_g)_{*1}v=q_{*g}gx$ is given by $\gamma(t)=ge^{tx}\cdot o$.
\item Translations $\tau_t$ along $\gamma$ are given by $\tau_t(p)=ge^{tx}g^{-1}\cdot p$.
\item Paralell transport along $\gamma$ is given by $P_0^t(\gamma)w= (D\tau_t)_{\gamma_0}w=(\ell_{e^{t \textrm{Ad}_g x}})_{*g\cdot o}w$.
\item For $x,y,z\in \mathfrak m$, if $X=q_{*g}(gx)\in T_{g\cdot o}M$ etc., the curvature at $p=g\cdot o\in M$ is given by
$$
R_p(X,Y)Z=-q_{*g}(g[[x,y],z])=-(\ell_g)_{*o}q_{*1}([[x,y],z]).
$$
\item If $x\in\mathfrak m$ then $\rho_t(p)=e^{tx}\cdot p$ is the flow of the unique Killing field $X$ such that $DX_p=\Gamma_p(X_p,-)$  and $X(o)=q_{*1}x$. 
\end{enumerate}
\end{teo}

\subsection{The cone $\Omega$ as a Cartan symmetric space of $\GO$}\label{conesym}

We are now in position to present the cone of positive elements of a JB-algebra $\V$ as a Cartan homogeneous symmetric space, by the action of the group $\GO$:

\begin{defi}
	The action $\GO \curvearrowright \Omega$ is given by the smooth map $A:\GO\times\Omega\to\Omega$ defined as an evaluation,
	\begin{equation*}
A(g,p)= g\cdot p = gp = g(p).
	\end{equation*}
We denote the quotient map $q: \GO \to \Omega$, that is $q(g) = g(1)$.  In particular $q_{*g}\dot{g}=\dot{g}(1)$.
\end{defi} 

\begin{rem}[$\Omega=\GO/\Aut$]\label{cociente} This action is transitive as for every $p\in \Omega$, $p = U_{p^{1/2}}(1)$.  Hence $\Omega=q(\GO)=\mathcal O(1)$, the orbit of $1\in \V$ by the action of the symmetric group $\GO$. Note that $K$, the stabilizer of the action at $1\in \V$ is exactly $\Aut$. Then $q$ is a smooth submersion, and we can identify $\Omega \simeq \GO / \Aut$ as Banach manifolds; moreover
\begin{equation*}
	T_1\Omega = \lie(\GO)/\lie(\Aut) = (\mathbb{L}\oplus D)/D \simeq \mathbb{L}, 
\end{equation*}
which is obvious as $\mathbb{L}$ is isomorphic to $\V$ via the isomorphism $v \mapsto L_v$.
\end{rem}

By combining two antiautomorphisms of the structure group, we obtain an involutive automorphism $\sigma$ in the structure group $\Str$ that preserves the open subgroup $\GO$:

\begin{lema}\label{overline}If $g\in \Str$, let $\sigma(g)=(g^*)^{-1}=U_{g1}^{-1}\,g$. Then 
\begin{enumerate}
\item $\sigma:\Str\to\Str$ is an automorphism with $\sigma^2=1$ and $\sigma(\GO)=\GO$. 
\item $\sigma(U_x)=U_x^{-1}$ for any $x\in\V$, and the fixed point subgroup of $\sigma$ is $K=\Aut$. 
\item Denote $\overline{H}=\sigma_{*1}H$ for $H\in \str=\mathbb L\oplus\der$, then $\overline{L_x+D}=-L_x+D=-(L_x+D)^{\dagger}$, thus 
$$
\mathbb L=\{H\in \str: \sigma_{*1}H=-H\}=\mathfrak m
$$
$$
\der=\{H\in \str: \sigma_{*1}H=H\}=\mathfrak k
$$
is the  decomposition of $\str$ by means of the involutive automorphism $\sigma$.
\end{enumerate}
\end{lema}
\begin{proof}
Clearly $\sigma^2=1$, on the other hand as it is a composition of two antiautomorphisms it follows that $\sigma(gh)=\sigma(g)\sigma(h)$. Now if $g\in \GO$ so is $g^{-1}$, and every $U$-operator preserves the positive cone $\Omega$ so $U_{g1} \in \GO$, hence $g^*$ preserves the cone, showing that $\sigma(g)$ preserves the cone, thus $\sigma(\GO)=\GO$. From the fundamental formula we have that $(U_x^*)^{-1} = U_x^{-1}$ and on the other hand if $k\in \Aut$ then $U_{k1}=U_1=1$, thus $\sigma(k)=k$. For the last assertion, note first that 
$$
\sigma(e^{tL_x})=\sigma(U_{e^{tx/2}})=U_{e^{-tx/2}}=e^{-tL_x},
$$
and differentiating at $t=0$ shows that $\sigma_{*1}L_x=-L_x$. Finally from $\sigma(e^{tD})=e^{tD}$ when $D\in\der$, we see that $\sigma_{*1}D=D$.
\end{proof}

We will now compute the affine conection $\nabla$ and spray $F$ derived by the structure of symmetric space of $\Str$. For that, we need to compute the symmetric product in $\Omega$ and in $T\Omega$, and from there we compute the spray. 

\bigskip

Let $x,y\in \Omega$, let $V,W\in T\Omega$; identiying $T_x\Omega\simeq \V$ we can identify $V=(x,v)$, $W=(y,w)$, where $v,w\in\V$.

\begin{prop}[Spray and connection of $\Omega$ as a symmetric space of $\GO$] Let $x,y\in \Omega$, let $V=(x,v), W=(y,w)\in T\Omega$. Then 
\begin{enumerate}
\item $\mu(x,y)=x\cdot y=U_x(y^{-1})$ in $\Omega$.
\item $\mu_*(V,W)=2U_{x,v}(y^{-1})-U_xU_y^{-1}w$ in $T\Omega$.
\item $F(V)=F_x(v)=U_v(x^{-1})$ is the spray of the symmetric structure $(\Omega,\mu)$.
\end{enumerate}
\end{prop}
\begin{proof}
Let $x = U_{x^{1/2}}(1)$, $y = U_{y^{1/2}}(1)$. Then the symmetric product is given by (\ref{symstr}):
\begin{equation*}
	\mu(x,y) = x \cdot y = U_{x^{1/2}} \sigma\left(U_{x^{1/2}}\right)^{-1} \sigma\left(U_{y^{1/2}}\right)(1) = U_{x^{1/2}}U_{x^{1/2}}^*U_{y^{-1/2}}^*(1) = U_x(y^{-1}).
\end{equation*}
The symmetric product in $T\Omega$ is defined by $\mu_*$ (Remark \ref{teme}). Consider $\alpha$, $\beta : I \to \Omega$ such that $\alpha(0) = x$, $\alpha'(0) = v$ and $\beta(0) = y$, $\beta'(0) = w$. Then
\begin{equation*}
\begin{aligned}
	\mu_*(v,w) &= \mu_{*(x,y)}(v,w) = \left(\mu(\alpha,\beta)\right)'(0) = \left(U_\alpha(\beta^{-1})\right)'(0) \\
	&= 2U_{\alpha(0),\alpha'(0)}(\beta(0)^{-1}) - U_{\alpha(0)}U_{\beta(0)^{-1}}(\beta'(0))=2U_{x,v}(y^{-1})- U_xU_y^{-1}(w).
\end{aligned}
\end{equation*}
We know that $F_x(v) = -(\Sigma_{v/2} \circ Z)_*(v)$, where $Z$ is the null section from $\Omega$ to $T\Omega$ (Theorem \ref{geosim}); then
\begin{equation*}
	\Sigma_{v/2} \circ Z(y) = \Sigma_{v/2}(0_{T_y\Omega}) = U_{x,v}(y^{-1}) - U_xU_y^{-1}(0) = U_{x,v}(y^{-1}).
\end{equation*}
Now, if $w$ belongs to $T_y\Omega$, let's calculate the differential of this function in $w$. Take $\alpha: I \to \Omega$ with $\alpha(0) = y$, $\alpha'(0) = w$. Then
\begin{equation*}
	\left(\Sigma_{v/2} \circ Z\right)_{*y}(w) = \left(\Sigma_{v/2} \circ Z(\alpha)\right)'(0) = \left(U_{x,v}(\alpha^{-1})\right)'(0) = -U_{x,v}U_y^{-1}w.
\end{equation*}
Taking $y = x$ and $w = v$, we have that
\begin{equation*}
	F_x(v) = -(\Sigma_{v/2} \circ Z)_*(v) = U_{x,v}U_x^{-1}v = U_v(x^{-1}),
\end{equation*}
where the last equality holds by the Shirshov-Cohn's theorem  (see Remark \ref{repre} below).
\end{proof}
 
 \begin{rem}[$\Str\subset Aut(\Omega)$] For each $g\in \Str$ we have that
$$
\mu(gx,gy)=U_{gx} (gy)^{-1} =gU_xg^{-1} U_{g1}U_{g1}^{-1}g(y^{-1})=gU_x(y^{-1})=g\mu(x,y)
$$
by \cite[Theorem VIII.2.5]{faraut}, thus $\Str\subset Aut(\Omega,\mu)=Aut(\Omega,F)$   (Theorem \ref{geosim}). We show below (Proposition \ref{kill3}) that the inner structure group $InnStr(\V)=\langle U_x\rangle_{x\in V\textrm{ invertible }}\subset \Str$ of $\V$ acts transitively on $\Omega$, giving the paralell transport along geodesics. For a full description of $Aut(\Omega,F)$, see the Appendix in \cite{chucrelle}.
 \end{rem}

\begin{rem}[Affine connection of the symmetric space $(\Omega,\mu)$]\label{gamaU} We first obtain the Christoffel operators
\begin{equation*}
	\Gamma_x(v,w) = \frac{1}{2}(F_x(v+w) - F_x(v) - F_x(w)) = \frac{1}{2}(U_{v+w} - U_v - U_w)(x^{-1}) = U_{v,w}(x^{-1}).
\end{equation*}
And the affine conection for $V,W\in \mathfrak X(\Omega)$ is then
\begin{equation*}
	\nabla_VW(x) = DW_x(V_x) - \Gamma_x(V_x,W_x) = DW_x(V_x) - U_{V_x,W_x}(x^{-1}).
\end{equation*}
This gives us the covariant derivative of a field $X$ along a curve $\gamma\subset \Omega$
\begin{equation*}
	\frac{DX}{dt} = \nabla_{\gamma'}X = \frac{dX}{dt} - U_{\gamma',X}(\gamma^{-1}).
\end{equation*}
\end{rem}

\begin{rem}[The relation between the sprays and connections of $\GO$ and of $\Omega$] 
Let $F^G$ denote the spray of the group $\GO$ (Definition \ref{Fstr}), and let us use $F^{\Omega}$ for the spray of $\Omega$ given in the previous proposition, $F^{\Omega}_p(v)=U_v(p^{-1})$. Recall that for horizontal vectors we have $F^G_g(gL_x)=g(L_x)^2$, now if $q:\GO\to\Omega$ is the quotient map and $p=U_{p^{1/2}}(1)=q(U_{p^{1/2}})\in \Omega$, then naming $g=U_{p^{1/2}}$ we have
\begin{align*}
F^{\Omega}_{q(g)}(q_{*g}(gL_x))& =F^{\Omega}_p( gL_x(1))=F^{\Omega}_p(U_{p^{1/2}}x)=U_{U_{p^{1/2}x}}(p^{-1})=U_{p^{1/2}}U_xU_{p^{1/2}}(p^{-1})\\
&=U_{p^{1/2}}U_x(1)=U_{p^{1/2}}(x^2)= g(x^2)=g(L_x)^2(1)=F^G_g(gL_x)(1)\\
&=q_{*g}(F^G_g(gL_x)).
\end{align*} 
Since $q_{*}$ is a linear map,  $D^2q_{*g}\equiv 0$ for any $g\in \GO$, and  the sprays are $q$-related, i.e.
$$
F^{\Omega}_{q(g)}(q_{*g}(V))=D^2q_{*g}(V,V)+q_{*g}F^G_g(V)
$$
for any horizontal vector $V\in \mathcal H_g=g\,\mathbb L\subset T\GO$. Keep in mind however, that this will also hold true for any quadratic spray in $\Str$ of the form (\ref{Fgeneral}) that vanishes in  $\mathbb L$, i.e. such that $B(L_x,L_x)=~0$.
\end{rem}

\begin{rem}[Horizontal vector fields] From the last  identity, or also by direct computation, it is not hard to see that if $X,Y\in \mathfrak X(\GO)$ are \textit{horizontal} vector fields, then for the push-forward vector fields $q_{*}X,q_{*}Y\in \mathfrak X(\Omega)$ and the respective affine connections, we have
$$
q_*(\nabla^G_X Y(g))=\nabla^{\Omega}_{q_{*}X}(q_{*}Y)(q(g))\qquad \forall\,g\in \GO.
$$
In particular horizontal geodesics of $\GO$ (which are of the form $\gamma(t)=ge^{t L_v}$ with $g\in \GO$, see Remark \ref{geoGO}) are mapped by $q$ to geodesics of $\Omega$, which are of the form 
$$
\alpha(t)=q(\gamma(t))=\gamma(t)(1)=ge^{tL_v}(1)=gU_{e^{tv/2}}(1)=g(e^{tv}).
$$
This is discussed with more detail in the next section.
\end{rem}

We can describe the geodesics in this space, from here $\nabla$ denotes the affine connection induced by the symmetric structure  $(\Omega,\mu)$ as described above.

\begin{prop}[Geodesics]
Let $x\in \Omega$ and $v\in \V = T_x\Omega$. We can choose $g = U_{x^{\frac{1}{2}}}$ so that $g(1) = x$. The  unique geodesic $\alpha$ of $(\Omega,\nabla)$  such that $\alpha(0) = x$, $\alpha'(0) = v$ is
$$	
\alpha(t) = ge^{tL_{g^{-1}v}}(1) = U_{x^{1/2}}e^{tL_{U_{x^{-1/2}}v}}(1)=U_{x^{1/2}}U_{\exp({\frac{t}{2}U_{x^{-1/2}}v})}(1)= U_{x^{1/2}}\exp({t\, U_{x^{1/2}}^{-1}v}).
$$
Thus the exponential map of the connection $\nabla$ is $\exp_x(v) =\alpha(1)= U_{x^{1/2}}\exp(U_{x^{1/2}}^{-1}v)$, with global smooth inverse for $x,y\in \Omega$ given by $\exp_x^{-1}(y) =  U_{x^{1/2}}\log{(U_{x^{1/2}}^{-1}y)}$.
\end{prop}
\begin{proof}
	We only need to check that $\alpha'' = F_{\alpha}(\alpha')$. We have that
	\begin{equation*}
	\begin{aligned}
		F_{\alpha}(\alpha') &= U_{\alpha'}(\alpha^{-1}) = U_{U_{x^{1/2}}(\exp(tU_{x^{-1/2}}v)\circ U_{x^{-1/2}}v)}((U_{x^{1/2}}(\exp(tU_{x^{-1/2}}v)))^{-1})\\
		&= U_{x^{1/2}}U_{\exp(tU_{x^{-1/2}}v)\circ U_{x^{-1/2}}v}U_{x^{1/2}}U_{x^{-1/2}}(\exp(-tU_{x^{-1/2}}v)) \\
		&=U_{x^{1/2}}U_{U_{x^{-1/2}}v}U_{\exp(tU_{x^{-1/2}}v)}(\exp(-tU_{x^{-1/2}}v))= U_{x^{1/2}}U_{U_{x^{-1/2}}v}(\exp(tU_{x^{-1/2}}v))\\
		&=U_{x^{1/2}}(U_{x^{-1/2}}v \circ(U_{x^{-1/2}}v \circ \exp(tU_{x^{-1/2}}v)))= \alpha'',
\end{aligned}
\end{equation*}
where we used repeatedly that operations are made in the strongly associative subalgebra generated by $U_{x^{-1/2}}v$.
\end{proof}

\begin{rem}\label{geoxy}
	Let $x,y\in  \Omega$, consider the path that joins $x,y$ inside $\Omega$ given by:
	\begin{equation}\label{geodxy}
		\alpha_{x,y}(t) = U_{x^{1/2}}\exp(t\, \log(U_{x^{-1/2}}y)).
	\end{equation}
	Note that this curve has initial point $x$ and initial speed $v=U_{x^{1/2}}\log{(U_{x^{-1/2}}y)}$, and the geodesic with these initial paramenters is
	\begin{equation*}
		\alpha(t) = U_{x^{1/2}}e^{tU_{x^{-1/2}}U_{x^{1/2}}\log{(U_{x^{-1/2}}y)}} =  U_{x^{1/2}}e^{t\log{(U_{x^{-1/2}}y)}} = \alpha_{x,y}(t).
	\end{equation*}
	So $\alpha_{x,y}$ is a geodesic that unites $x$ and $y$. Moreover, it is unique. Now note that for each $z\in \V$, we have $e^{tL_z}(1)=U_{e^{tz/2}}(1)=e^{tz}$, thus for $z=\log(U_{x^{-1/2}}y)$ we can rewrite
$$
\alpha_{x,y}(t)=U_{x^{1/2}}e^{tL_z}(1)=U_{x^{1/2}}e^{t z}.
$$
So $\alpha_{x,y}$ is the image by the quotient map $q:\GO\to \Omega$ of the path $t\mapsto U_{x^{1/2}}e^{tL_z}$ in $\GO$.
\end{rem}

By Remarks \ref{quadrspr} and \ref{gamaU}, a vector field $\eta$ in $T\Omega$ parallel along $\gamma\subset\Omega$ must be a solution of the differential equation
\begin{equation*}
	\eta' = \Gamma_{\gamma}(\gamma',\eta) = U_{\gamma',\eta}(\gamma^{-1}).
\end{equation*}
This equation with initial value $\eta(0) = w$ has a unique solution. If the path $\gamma$ is a geodesic we will show how to compute the parallel transport along it.

\begin{prop}[Parallell transport in $\Omega$]\label{paraomega}
	Parallel transport along the geodesic $\alpha_{x,y}$ is given by 
$$
P_s^{s+t}(\alpha_{x,y}) =U_{x^{1/2}}U_{\exp(t/2\log{(U_{x^{-1/2}}y)})}U_{x^{-1/2}}.
$$
and in particular paralell transport along any geodesic is implemented by the Inner Structure group.
\end{prop}
\begin{proof}
By the above remark, if we name $z=\log(U_{x^{-1/2}}y)\in \V$, we can write $\alpha_{x,y}(t)=ge^{tT}\cdot o$ where $g=U_{x^{1/2}}$ and $T=L_z\in \GO$. Hence by Theorem \ref{geosimcar}, the one parameter-group of automorphisms
	\begin{equation*}
 		\tau_t(p) = ge^{tT}g^{-1}(p)=U_{x^{1/2}}e^{tL_z}U_{x^{-1/2}}(p)=U_{x^{1/2}}U_{e^{tz/2}}U_{x^{-1/2}}(p)
	\end{equation*}
gives translation along  $\alpha$. Moreover, by the same theorem, parallel transport along $\alpha_{x,y}$ is given by $P_s^{s+t}(\alpha_{x,y}) = (\tau_t)_{*\alpha_{x,y}(s)}$. Since $p\mapsto \tau_t(p)$ is linear, its differential is itself and the proof is finished.
\end{proof}

From Remark \ref{curvatura2}, the formula for the curvature tensor in our symmetric space is
\begin{equation*}
\begin{aligned}
	R_p(V,W)Z &= \Gamma_p(V,\Gamma_p(W,Z)) - \Gamma_p(W,\Gamma_p(V,Z)) \\
	&= U_{V,U_{W,Z}(p^{-1})}(p^{-1}) - U_{W,U_{V,Z}(p^{-1})}(p^{-1}).
\end{aligned}
\end{equation*}
This can be further rewritten as follows: 

\begin{prop}[Curvature in $\Omega$] Let $p=U_{p^{1/2}}(1)\in \Omega$, write $v=U_{p^{-1/2}}V$, $w=U_{p^{-1/2}}W$, $z=U_{p^{-1/2}}Z$. Then
$$
R_p(V,W)Z=-U_{p^{1/2}}[L_v,L_w](z)=U_{p^{1/2}}(w\circ(v\circ z)-v\circ (w\circ z)),
$$
and in particular $R_p(V,W)V=U_{p^{1/2}}(U_v-(L_v)^2)w$.
\end{prop}
\begin{proof}
Let $g=U_{p^{1/2}}(1)$ and note that $V=U_{p^{1/2}}v=U_{p^{1/2}}L_v(1)=q_{*g}(gL_v)$ and likewise with $W,Z$. We recall that $\mathfrak m$, the eigenspace for $\lambda=-1$ of $\sigma_{*1}$ is exactly $\mathfrak m=\mathbb L$. Thus by Theorem \ref{geosimcar}, and noting that $[L_v,L_w](1)=vw-wv=0$, we have
\begin{align*}
R_p(V,W)Z&=-U_{p^{1/2}}[[L_v,L_w],L_z](1)= -U_{p^{1/2}}[L_v,L_w](z)\\
&=-U_{p^{1/2}}(v\circ (w\circ z)-w\circ(v\circ z)),
\end{align*}
where we used that $(U_v-(L_v)^2)(w)=((L_v)^2-L_{v^2})(w)=v^2\circ w-v\circ (v\circ w)$.
\end{proof}

\begin{rem}[Killing fields in $\Omega$]We now discuss Killing fields in $(\Omega,\nabla)$, see also the Appendix in \cite{chucrelle}. Given $z\in \lie(\GO) = \str$, consider $\rho_t(p) = e^{tz}(p)$, and define $X(p) = \rho'_0(p)=z(p)$. As $\rho_t$ is an one-parameter group, $X$ is a vector field with flow $\rho_t$. To see that $X$ is a Killing field, we need to show that $\rho_t$ belongs to $Aut(\Omega,\nabla)$, which in the case of symmetric spaces is equal to $Aut(\Omega,\mu)$ by Theorem \ref{geosim}. Then
\begin{equation*}
\begin{aligned}
	\mu(e^{tz}(x),e^{tz}(y)) &= \mu(e^{tz}U_{x^{1/2}}(1),e^{tz}U_{y^{1/2}}(1)) = e^{tz}U_{x^{1/2}}\sigma(e^{tz}U_{x^{1/2}})^{-1} \sigma(e^{tz}U_{y^{1/2}})(1) \\
	&= e^{tz}U_xe^{-t\sigma_*z} e^{t\sigma_*z}U_{y^{-1/2}}(1)  = e^{tz}U_x(y^{-1})= e^{tz}\mu(x,y),
\end{aligned}
\end{equation*}
so $X$ is a Killing field. Moreover, if $z=L_x+D$, then $X(1) = z(1)=xp$ and $DX_1 = \Gamma_1(z(1),-) = L_{z(1)}=L_x$. We know that given $p\in \Omega$, the value of a Killing field $V$ and its derivative in $p`$ determine $V$, so we have all the Killing fields. By means of  Theorem \ref{geosimcar}, we can further state that
\end{rem}

\begin{prop}\label{kill3}
The unique Killing field $X$ in $(\Omega,\nabla)$ with $X(1)=x\in\V=T_1\Omega$ and $\nabla X(1)=0$ (i.e. $DX_1=L_x=\Gamma_1(x,-)$) is given by $X(p)=p\circ x$. It's flow is $\rho_t(p)=e^{tL_x}(p)=U_{e^{tx/2}}(p)$.
\end{prop}

\begin{rem}[The case of a special Jordan algebra]
	In the case of a special Jordan algebra, this results are maybe well-known: as $U_xy = xyx$ with the associative product, then straightforward computations show that
	\begin{align*}
		&\mu(x,y) = x \cdot y = xy^{-1}x  \qquad \mu_*(v,w) = v \cdot w = xy^{-1}v + vy^{-1}x -xy^{-1}wy^{-1}x\\
		&F_x(v) = vx^{-1}v \qquad \qquad\quad\;\;\Gamma_x(v,w) = 1/2(vx^{-1}w + wx^{-1}v)\\
&\frac{DX}{dt} = \frac{dX}{dt} - \frac{1}{2}(\dot{\gamma}\gamma^{-1}X + X \gamma^{-1}	\dot{\gamma})\qquad\exp_x(v) = x^{1/2}e^{x^{-1/2}vx^{-1/2}}x^{1/2}.
\end{align*}
The geodesic joining $x,y$ is $\alpha_{x,y}(t) = x^{1/2}e^{t \log(x^{-1/2}yx^{-1/2})}x^{1/2}$. Parallel transport along the geodesic joining $x,y$ is then
	\begin{equation*}
		P_s^{s+t}(\alpha_{x,y})(v) = x^{1/2}(x^{-1/2}yx^{-1/2})^{t/2}x^{-1/2}vx^{-1/2}(x^{-1/2}yx^{-1/2})^{t/2}x^{1/2}.
	\end{equation*}
If we write $V=p^{1/2}vp^{1/2}$, and likewise for $W,Z$, then the curvature tensor at $p\in\Omega$ is given by
	\begin{equation*}
		R_p(V,W)Z = \frac{1}{4}\, p^{1/2}  [[v,w],z]p^{1/2}.
	\end{equation*}
The unique Killing field $X$ in $\Omega$ with $X(1)=x\in \V$ and $DX_1=L_x=\Gamma_1(x,-)$ is given by  $X(p)=1/2(xp+px)$ and its flow is $\rho_t(p)=e^{tx/2}pe^{tx/2}$.
\end{rem}

\section{Finsler metric and geodesic distance in $\Omega$}\label{finsleromega}

The positive cone $\Omega$ carries a natural Finsler metric, where the Finsler norm at each tangent space $T_p\Omega\simeq \V$ is defined as
$$
\|v\|_p=\|U_{p^{-1/2}}v\|=\|U_{p^{1/2}}^{-1}v\|.
$$
This structure is discussed in detail in \cite{kaupmeier}, \cite{neeb} and more recenlty in \cite{chu}.   Consider the rectifiable length 
$$
\Length_{\Omega}(\gamma) = \int_a^b\|\dot{\gamma}(t)\|_{\gamma(t)}dt,
$$
and its rectifiable distance
	\begin{equation*}
		\dist_{\Omega}(p,q) = \inf_{\gamma}\Length_{\Omega}(\gamma),
	\end{equation*}
	where the infimum is taken from all piecewise $C^1$  paths $\gamma\subset \Omega$ joining $p,q$. This defines a metric in $\Omega$, and since the Finsler norm is compatible, the manifold topology of $\Omega$ (the norm topology of $\V$) is the same as the topology induced by this metric in $\Omega$ (see \cite[Section 1]{chucrelle} and \cite[Proposition 12.22]{upmeier}).

\begin{rem}[$\GO$-invariance]\label{goinv} The Finsler metric is invariant for the action of the the group $\GO$ in $T\Omega$: if $g\in\GO$ and $p\in\Omega$ then
$$
U_{(gp)^{-1/2}}gU_{p^{1/2}}(1)=U_{(gp)^{-1/2}}(gp)=1,
$$
and since $k=U_{(gp)^{-1/2}}gU_{p^{1/2}}\in \GO$, it must be that $k\in \Aut$ \cite[Proposition 3.22]{larluna1}. Hence 
$$
U_{(gp)^{-1/2}}g=kU_{p^{-1/2}} \textrm{ for some }k\in \Aut,
$$
and then
$$
\|gv\|_{gp}=\|U_{gp}^{-1/2}gv\|=\|kU_{p^{-1/2}}v\|=\|U_{p^{-1/2}}v\|=\|v\|_p,
$$
since every automorphism of $\GO$ is an isometry \cite[Proposition 3.22]{larluna1}. Hence both the length of paths and the distance are invariant for the action of the group $\GO$. 
\end{rem}

Moreover, since paralell transport along geodesics is implemented by the Inner Structre group (Proposition \ref{paraomega}), and this implementation is the same that moves the base point, it is then apparent that

\begin{coro}\label{paraiso}
Paralell transport along geodesics is isometric for the Finsler metric in $\Omega$. In particular geodesics of the connection $\alpha(t) = U_{p^{1/2}}\exp(tU_{p^{1/2}}^{-1}v)$ have constant speed
$$
\|U_{\alpha_t^{-1/2}}\alpha_t'\|=\|\alpha_t'\|_{\alpha_t}=\|P_0^t(\alpha)\alpha_0'\|=\|\alpha_0'\|_{\alpha_0}=\|U_{p^{-1/2}}v\|.
$$
\end{coro}


\smallskip

The speed at the Lie algebra of any path can be expressed in terms of the $L$ operators as follows:

\begin{lema}\label{velolie}Let $\Gamma\subset \V$ be a piecewise smooth path and let $\gamma=e^{\Gamma}\subset \Omega$. Then
$$
U_{\gamma^{-1/2}}\gamma' = \{G(\ad L_{\Gamma})L_{\Gamma\,'}\}(1)
$$
where $G$ is the real analytic map $G(\lambda)=\frac{\sinh(\lambda)}{\lambda}$. 
\end{lema}
\begin{proof}
Note that $U_{\gamma^{-1/2}} = U_{e^{-\Gamma}/2} = e^{-L_{\Gamma}}$; moreover, $\gamma = e^{\Gamma} = U_{e^{\Gamma/2}}(1)$. Then, considering the map $F(\lambda) = \frac{1-e^{-\lambda}}{\lambda}$, $F(0) = 1$, we have that
	\begin{equation*}
	\begin{aligned}
		\gamma' &= (U_{e^{\Gamma/2}}(1))' = (U_{e^{\Gamma/2}})'(1) = (e^{L_{\Gamma}})'(1) = 
		\{e^{L_{\Gamma}}F(\ad L_{\Gamma})(L_{\Gamma'})\}(1)
	\end{aligned}
		\end{equation*}
	due to the well-known formula for the differential of the exponential map in associative Banach algebras (see for instance \cite[Remark 23]{larineq}). Then,
	\begin{equation*}
		U_{\gamma^{-1/2}}\gamma' = e^{-L_{\Gamma}}e^{L_{\Gamma}}F(\ad L_{\Gamma})L_{\Gamma\,'}(1)=F(\ad  L_{\Gamma})L_{\Gamma\,'}(1).
	\end{equation*}
	As for every $x$, $y$ in $\V$ we have that $[L_x,L_y]$ is a derivation, we have that the odd terms of $F(\ad L_x)L_y$ disappear if we evaluate in $1$. Moreover,
	\begin{equation*}
		F(\lambda) = \frac{1-e^{-\lambda}}{\lambda} = \frac{1 - \cosh(\lambda) + \sinh(\lambda)}{\lambda},
	\end{equation*}
	so the even terms of $F(\lambda)$ are the terms of $G(\lambda) = \frac{\sinh(\lambda)}{\lambda}$. Then
	\begin{equation*}
		U_{\gamma^{-1/2}}\gamma' = G(\ad L_{\Gamma})L_{\Gamma\,'}(1).\qedhere
	\end{equation*}
\end{proof}

All the previous remarks and results of this section still hold in place if we replace the order norm of $\V=T_1\Omega$ with any equivalent $\GO$-invariant norm in $\V$.

\begin{rem}\label{repre}Let $C(x,y)$ be the closed subalgebra of a JB-algebra generated by two elements $x,y$ and  $1$. By  the Shirshov–Cohn Theorem \cite[Proposition 2.1]{wright2}, it is isometrically isomorphic to a Jordan algebra of self-adjoint operators on a complex Hilbert space. In other words, it can be represented isometrically into a $C^*$ algebra with its special Jordan product. 
\end{rem}

\begin{lema}\label{rep1} Let $x,y\in \V$, let $\pi:C(x,y)\to \bh$ be any isometric representation of the closed $JB$-algebra generated by $1,x,y$ into a Hilbert space. Then for $G(\lambda)=\frac{\sinh(\lambda)}{\lambda}$ we have
$$
\pi\left(G(\ad L_x)(L_y)(1)\right)=G\left(\ad \frac{\pi(x)}{2}\right)\pi(y).
$$
\end{lema}
\begin{proof}
 It is well-known that for associative Banach algebras, for every pair of elements $T$ and $S$
	\begin{equation*}
		\frac{\sinh(\ad T)}{\ad T}(S) = \int_0^1e^{(2s-1)T}Se^{(1-2s)T}ds,
	\end{equation*}
see for instance \cite[Remark 23]{larineq}. We apply this formula in $\bv$ with $T =L_x$ and $S = L_y$ for $x,y\in\V$. If we call $I\in \V$ the evaluation of this operator in $v=1$, we get
\begin{align*}
		I &= \frac{\sinh(\ad L_x)}{\ad L_x}(L_y)(1) = \int_0^1e^{(2s-1)L_x}L_ye^{(1-2s)L_x}ds(1)\\
		&=\int_0^1e^{(2s-1)L_x}L_ye^{(1-2s)L_x}(1)ds = \int_0^1e^{(2s-1)L_x}L_yU_{e^{(1/2-s)x}}(1)ds\\
		&=\int_0^1e^{(2s-1)L_x}L_ye^{(1-2s)x}ds = \int_0^1U_{e^{(s-1/2)x}}\left(y\circ e^{(1-2s)x}\right)ds.
\end{align*}
	Note that the integrand of $I$ belongs to $C(x,y)$ for every $s \in [0,1]$. Then, so does $I$. As the integrand belongs to $C(x,y)$, the representation $\pi$ commutes with the integral. Then, applying $\pi$ to $I$ we have
	\begin{equation*}
	\begin{aligned}
		\pi(I) &= \pi\left( \int_0^1U_{e^{(s-1/2)x}}\left(y\circ e^{(1-2s)x}\right)ds\right)= \int_0^1U_{e^{(s-1/2)\pi(x)}}\left(\pi(y)\circ e^{(1-2s)\pi(x)}\right)ds \\
		&= \int_0^1\frac{1}{2}e^{(s-1/2)\pi(x)}\left(\pi(y)e^{(1-2s)\pi(x)} +e^{(1-2s)\pi(x)} \pi(y) \right)e^{(s-1/2)\pi(x)}ds\\
		&= \frac{1}{2}\int_0^1\left(e^{(2s-1)\frac{\pi(x)}{2}}\pi(y)e^{(1-2s)\frac{\pi(x)}{2}} + e^{(2s-1)\frac{-\pi(x)}{2}} \pi(y) e^{(1-2s)\frac{-\pi(x)}{2}}\right)ds\\
		&= \frac{1}{2}\left(\frac{\sinh(\ad(\frac{\pi(x)}{2}))}{\ad(\frac{\pi(x)}{2})}+\frac{\sinh(\ad(\frac{-\pi(x)}{2}))}{\ad(\frac{-\pi(x)}{2})} \right)\pi(y)=\frac{\sinh(\ad(\frac{\pi(x)}{2}))}{\ad(\frac{\pi(x)}{2})}\pi(y),
	\end{aligned}
	\end{equation*}
	where the last equality holds because $\sinh(\lambda)/\lambda$ is an even map. 	
	\end{proof}

\begin{coro}\label{espectro} Let $\gamma_t=e^{\Gamma_t}$ be a smooth path in $\Omega$. For each $t\in [a,b]=Dom(\gamma)$, let $\pi_t$ be any isometric representation of the JB-algebra generated by $1,\Gamma_t,\Gamma_t'\in V$ into a special Jordan algebra. Then 
$$
\pi_t(U_{\gamma_t^{-1/2}}\gamma_t')=  G(\ad \pi_t \Gamma_t/2)\pi_t(\Gamma_t').
$$
In particular both operators have the same spectrum, and if the eigenvalues of $U_{\gamma_t^{-1/2}}\gamma_t'$ are isolated, their multiplicity is the same for both operators.
\end{coro}
\begin{proof}
We apply the the previous lemma to $x=\Gamma_t,y=\Gamma_t'$, and we combine it with Lemma \ref{velolie} to obtain the equality. Then the assertions on the spectrum follow because the spectrum (and its multiplicity) is invariant for $\pi$.
\end{proof}

With this it is immediate the optimality of the connection geodesics for the Finsler invariant metric in $\Omega$. In \cite{nussbaum} the reader can find a completely different proof. See also \cite[Remark 6.6]{neeb} for the proof in the infinite dimensional setting, and also \cite{corach} for the particular setting of $C^*$-algebras. We will show in the next section how our new proof enables the generalization of this result to any metric in $\Omega$ induced by a symmetric norm in $\V$. 

\begin{teo}\label{minisup}
	The geodesic $\alpha(t) = U_{p^{1/2}}\exp(tU_{p^{1/2}}^{-1}v)$ is minimizing for $\dist_{\Omega}$ in $\Omega$, i.e.
	$$
	\Length_{\Omega}(\alpha)=\|U_{p^{-1/2}}v\|=\dist_{\Omega}(\alpha(0),\alpha(1)).
	$$
\end{teo}
\begin{proof}
	We will prove the theorem for geodesics $\alpha$ such that $\alpha(0) = 1$. As the metric is invariant for the transitive action of the group $\GO$, this will prove the result for all geodesics. So assumme that $\alpha(t) = e^{tv}$ for some $v\in \V$. Take $\gamma = e^{\Gamma}$ a smooth path in $\Omega$. By the previous results, and using that each $\pi_t$ is isometric, we have
\begin{align*}
\|\gamma'\|_{\gamma} & = \|U_{\gamma^{-1/2}}\gamma'\|=\|G(\ad L_{\Gamma})L_{\Gamma\,'}(1)\| = \|\pi(G(\ad L_{\Gamma})L_{\Gamma'}(1))\|\\
& =\|\frac{\sinh(\ad\pi(\Gamma)/2))}{\ad\pi(\Gamma)/2}\pi(\Gamma')\|.
\end{align*}
Now, since $X=\pi(\Gamma)/2$ is a self-adjoint operator of a $C^*$-algebra, we can apply \cite[Remark 23]{larineq}, obtaining 
$||\gamma'||_{\gamma}\ge  ||\Gamma'||$. Finally, if $\gamma= e^{\Gamma}\subset \Omega$ joins $1,e^v$, it must be that $\Gamma(0)=0$, $\Gamma(1)=v$, and then 
	\begin{equation*}
		\Length_{\Omega}(\gamma) = \int||\gamma'||_{\gamma}\geq\int||\Gamma'|| = \Length_{\V}(\Gamma) \geq ||v|| = \Length_{\Omega}(\alpha),
	\end{equation*}
since in any Banach space $(\V,\|\cdot\|)$ the length of any smooth path joinin $0,v$ is at least $\|v\|$. This proves that $\alpha$ is shorter than any other piecewise smooth path joining the given endpoints, and we are done.
\end{proof}

\begin{rem}[Thompson's part metric]
If $p,q\in \Omega$, then by Remark \ref{geoxy} and the previous theorem
$$
\dist_{\Omega}(p,q)=\|\log(U_{p^{-1/2}}q)\|=\|\log(U_{q^{-1/2}}p)\|,
$$
this is Thompson's part metric for a cone described by Nussbaum \cite{nussbaum}. It was proved by Lawson and Lim that this distance among two geodesics in $\Omega$ is a convex function of the time parameter \cite{ll}; theirs is a generalization of the theorem by Corach, Porta and Recht for $C^*$-algebras obtained in \cite{corach2}. For a discussion on the conditions for the uniqueness of geodesics among $p,q\in\Omega$, and the isometries of this metric, see Lemmens, Roelands and Wortel papers \cite{bas, bas2} and the references therein.
\end{rem}

\subsection{Symmetric Gauge norms in $\Omega$}

If $\V$ is finite dimensional, we can consider a symmetric gaguge invariant function $\phi:\mathbb R^n\to\mathbb R_{\ge 0}$, which is permutation invariant (see \cite[Chapter IV]{bhatia}). Let $s_k(x)\in\mathbb R_{\ge 0}$ be the modulus of the eigenvalues of $x$, counted with multiplicity and in decreasing order. Then we can define an equivalent norm in $\V$ by means of 
\begin{equation}\label{sgn}
\|x\|_{\phi}=\phi(s_1(x),s_2(x),\dots,s_n(x)).
\end{equation}
Clearly $\|\cdot\|_{\phi}$ is positive, non-degenerate and homogeneous. To prove that it is sub-additive, take a representation $\pi$ of the closed JB-algebra generated by $x,y\in\V$ (Remark \ref{repre}) into $\mathbb C^n$. Then the representation preserves the eigenvalues and their multiplicity, hence
\begin{align*}
\|x+y\|_{\phi}& =\phi(s_i(x+y))=\phi(s_i(\pi(x+y))=\phi(s_i(\pi(x)+\pi(y))\\
&\le \phi(s_i(\pi(x)))+\phi(s_i(\pi(y)))=\phi(s_i(x))+\phi(s_i(y))\\
&=\|x\|_{\phi}+\|y\|_{\phi},
\end{align*}
where the inequality is due to the majorization inequalities  \cite[Theorem IV.2.1]{bhatia}. 

\smallskip

We now can give the cone $\Omega$ the Finsler structre induced by the action of $\GO$ and this gauge norm: define
$$
|v|_p= \|U_{p^{-1/2}}v\|_{\phi},\quad p\in\Omega, \;v\in T_p\Omega\simeq \V.
$$
If $g\in\GO$ and $p\in\Omega$ recall from Remark \ref{goinv} that $k=U_{(gp)^{-1/2}}gU_{p^{1/2}}\in \Aut$. Hence
$$
|gv|_{gp}=\|U_{gp}^{-1/2}gv\|_{\phi}=\|kU_{p^{-1/2}}v\|_{\phi}=\|U_{p^{-1/2}}v\|_{\phi}=|v|_p,
$$
since every automorphism of $\GO$ preserves the spectrum and its multiplicity. Hence the length of paths $\Length_{\phi}$ and the rectifiable distance $\dist_{\phi}$ are again invariant for the action of the group $\GO$. 

\medskip

Let $\alpha(t)=U_{p^{1/2}}\exp(tU_{p^{-1/2}}v)$ be the geodesic of the symmetric space structure of $\Omega$. With the same proof as Corollary \ref{paraiso}, paralell transport along $\alpha$ is an isometry of $(\Omega,\dist_{\phi})$. We now prove that geodesic are also minimizing for the gauge norm distance:

\begin{teo}\label{minig}
Let $\V$ be a finite dimensional Jordan algebra and consider any symmetric gauge norm in $\V$ as in (\ref{sgn}). Then the geodesics of the symmetric space connection of $\Omega$ are minimizing, i.e.
	$$
	\Length_{\phi}(\alpha)=\|U_{p^{-1/2}}v\|_{\phi}=\dist_{\phi}(\alpha(0),\alpha(1)).
	$$
	Moreover, if the symmetric gauge norm is strictly convex, then geodesics are the unique minimizing paths.
\end{teo}
\begin{proof}
Again, by the invariance of the action of $\GO$ it suffices to prove that $\alpha(t)=e^{tv}$ is minimizing. Let $\gamma=e^{\Gamma}\subset \Omega$ be any piecewise $C^1$ path with $\Gamma(0)=0,\Gamma(1)=v$. By Corollary \ref{espectro}, we  know that for each $t$, the element $U_{\gamma_t^{-1/2}}\gamma_t'$ and the operator $G(\ad \pi_t \Gamma_t/2)\pi_t(\Gamma_t')$ have the same eigenvalues, and with the same multiplicity. Hence 
\begin{align*}
|\gamma_t'|_{\gamma_t} & =\|U_{\gamma_t^{-1/2}}\gamma_t'\|_{\phi}=\phi( s_i(G(\ad \pi_t \Gamma_t/2)\pi_t(\Gamma_t')))=\|G(\ad \pi_t \Gamma_t/2)\pi_t(\Gamma_t')\|_{\phi}
\end{align*}
where the norm on the right is the symmetric gauge norm induced by $\phi$ in $M_n(\mathbb C)=\B(\mathbb C^n)$, of the self-adjoint operator  $G(\ad \pi_t \Gamma_t/2)\pi_t(\Gamma_t')$. Now again by \cite[Remark 23]{larineq}, since $\pi_t\Gamma_t$ is self-adjoint, we have that 
$$
\|G(\ad \pi_t \Gamma_t/2)\pi_t(\Gamma_t')\|_{\phi}=\|\frac{\sinh(\ad\pi(\Gamma_t)/2))}{\ad\pi(\Gamma_t)/2}\pi(\Gamma_t')\|_{\phi} \ge ||\Gamma_t'||_{\phi}.
$$
Then as in the proof of Theorem \ref{minisup}
\begin{equation*}
		\Length_{\phi}(\gamma) = \int |\gamma'|_{\gamma}\geq\int||\Gamma'||_{\phi} = \Length_{\phi}(\Gamma) \geq ||v||_{\phi} =\Length_{\phi}(\alpha),
	\end{equation*}
showing that the geodesic is minimizing. Now if $\Length_{\phi}(\gamma) = \Length_{\phi}(\alpha)$, then in particular it must be that $\Length_{\phi}(\Gamma) = ||v||_{\phi}$, and if the norm is strictly convex, this is only possible if $\Gamma(t)=tv$. Therefore $\gamma(t)=e^{tv}$ is a geodesic.
\end{proof}

A generalization of these minimality results for infinite dimensional $JBW$-algebras and their symmetric gauge norms will appear elsewhere.

\section{Finsler metrics in $\GO$}\label{fgo}

We will now turn to define a Finsler norm in $\GO$.

\begin{defi}
	If $H\in \lie(\GO)=\str \subset \bv$, consider $\|H\|=\|H\|_{\infty}$ the supremum norm of $H$ as a linear operator acting on $\V$. If $H\in T_g\GO = g\str$, we define
	\begin{equation*}
	\|H\|_g = \|g^{-1}H\|.
	\end{equation*}
\end{defi}
This is a left-invariant Finsler norm in $\GO$. Moreover, since every $k\in \Aut$ is an isometry, and we can identify $T_g\GO\cdot k=T_{gk}\GO$, we have 
$$
\|Hk\|_{gk}=\|k^{-1}g^{-1}Hk\|=\|g^{-1}H\|=\|H\|_g,
$$
thus this norm is right-invariant (hence bi-invariant) for the action of the group $\Aut$. 

\begin{defi}
	The length of a piecewise $C^1$ path $\alpha: [a,b] \to \GO$ is
	\begin{equation*}
	 	\Length_{\GO}(\alpha) = \int_a^b \|\dot{\alpha}(t)\|_{\alpha(t)} dt.
	\end{equation*} 
and the distance among  $g,h\in \GO$ is given by
	\begin{equation*}
		\dist_{\GO}(g,h) = \inf_{\alpha}{\Length_{\GO}(\alpha)},
	\end{equation*}
	where the infimum is taken over all the paths $\alpha\subset \GO$ joining $g$ and $h$. Note that this distance is left-invariant as for every smooth path $\alpha$ the map $g\alpha$ is also  smooth. Moreover, both the length and the distance are bi-invariant for the action of the subgroup $\Aut$. Using the continuity of the inverse in $\glv$ and of multiplication in $\bv$, and the fact that $\GO$ is an embedded submanifold of $\glv$, it is not hard to check that the Finsler structure is compatible with the manifold structure (see \cite[Proposition 12.22]{upmeier}). Thus the topology of the rectifiable distance $\dist_{\GO}$ matches the topology of the Banach-Lie group $\GO$ (which is the norm topology of $\bv$ by Remark \ref{goemb}).
\end{defi}

As mentioned in Remark \ref{totageo}, the geodesics of $\Aut$ for the quadratic spray given in Definition \ref{Fstr} are the one-parameter groups $t\mapsto ke^{tD}$, with $k\in\Aut$ and $D\in \der$. We now show that these are minimizing path for the (restricted) Finsler structure of $\Aut$.

\begin{teo}\label{mini}
Let $k\in\Aut$, $D\in \der$ such that $\|D\|<\pi/2$. Consider the goedesic  $\delta(t)=ke^{tD}$, then $\delta$ is minimizing in $\Aut$, i.e.
$$
\dist_{\Aut}(k,ke^D)=\|D\|=\Length(\delta).
$$
Moreover, if $\gamma\subset \Aut$ is any piecewise $C^1$ path joining $k,ke^D$ such that $\Length(\Gamma)=\|D\|$, then for any unit norm functional $\varphi\in \der^*$ such that $\varphi(D)=\|D\|$, we have that $\gamma_t^{-1}\gamma_t'$ sits inside the same face of the sphere $F_{\varphi}=\varphi^{-1}(\|D\|)$ for all $t$. In particular, if the norm is smooth at $D$, then the unique geodesic is $\delta$.
\end{teo}
\begin{proof}By \cite[Theorem 3.17]{larluna1}, the exponential map of $\bv$ is a diffeomorphisms of the ball of radius $\pi/2$ in $\der$ onto its image in $\Aut$, which is an open neighbourhood of $Id\in\Aut$. Then the minimality of the one-parameter group follows from \cite[Theorem 4.11 and Theorem 4.22]{largrupos}.
\end{proof}

\subsection{The quotient distances in $\Omega=\GO/\Aut$}

With the above distance in $\GO$, we can define a quotient distance in $\Omega$. Given two elements in the cone, the distance between them will be the distance between their respective fibers.

\begin{defi}
	Take $x,y\in\Omega$ and take $g,h\in \GO$ such that $g(1) = x$, $h(1) = y$. Define 
	\begin{equation*}
		d'(x,y) = \dist_{\GO}(g\Aut,h\Aut) = \inf_{k_1,k_2 \in \Aut} \dist_{\GO}(gk_1,hk_2).
	\end{equation*}
\end{defi}

\begin{prop}
	$d'$ is a well defined $\GO$-invariant distance in  $\Omega$.
\end{prop}
\begin{proof}
	If $g_1(1) = g_2(1)$, then $g_1\Aut = g_2\Aut$, then $d'$ is well defined.  As $\dist_{\GO}$ is  right-invariant for the action of $\Aut$, we have
	\begin{equation*}
		d'(x,y) = \dist_{\GO}(g\Aut,h\Aut) = \dist_{\GO}(g,h\Aut) \leq \dist_{\GO}(g,h).
	\end{equation*}
	The distance from a point to a closed set is positive, so $d'$ is a distance. Moreover, as $\dist_{\GO}$ is also left-invariant, for any $f\in\GO$ we have
	\begin{equation*}
	\begin{aligned}
		d'(fx,fy) &= \dist_{\GO}(fg,fh\Aut) = \dist_{\GO}((fh)^{-1}fg,\Aut) \\
		&= \dist_{\GO}(h^{-1}g,\Aut) = \dist_{\GO}(g,h\Aut) = d'(x,y),
	\end{aligned}
	\end{equation*}
	so $d'$ is invariant for the action of $\GO$.
\end{proof}

We will soon show that this distance $d'$ is in fact equal to the rectifiable distance $\dist_{\Omega}$ defined in Section \ref{finsleromega} by means of the Finsler norms.

\subsubsection{The metric of $\Omega$ as a quotient Finsler metric}

We now consider the following invariant norm in $T\Omega$:
\begin{defi} Let $Z\in T_g\GO=g\str$, let $Z(1)=z\in \V=T_{g(1)}\Omega$. Consider the \textit{quotient Finsler norm}, which is the quantity at $g(1)\in\Omega$ given by 
	\begin{equation*}
		\|z\|_{g1} = \inf_{D \in \der}\|Z - gD\|_g = \textrm{dist}_{\|\cdot\|_{\infty}}(g^{-1}Z, \der).
	\end{equation*}
\end{defi}

\begin{lema}The quotient norm in $T\Omega$ is equal to the Finsler norm $\|v\|_p=\|U_{p^{-1/2}}v\|$ of Section \ref{finsleromega}, in particular it is a good definition and it does not depend on the $Z$ such that $Z(1)=z$.
\end{lema}
\begin{proof}
Let $g\in \GO$, let $Z=g(L_x+D_0)\in T_g\Str$. Then $Z(1)=gx$ thus
$$
\|Z-gD\|_g=\|L_x+D_0-D\|\ge  \|(L_x+D_0-D)(1)\|=\|x\|=\|g^{-1}Z(1)\|
$$
and taking infimum over $D\in \der$ we see that $\|Z(1)\|_{g1}\ge \|g^{-1}Z(1)\|$. On the other hand by considering the special case of $D=D_0$ in the infimum we see that 
$$
\|Z(1)\|_{g1}\le \|g(L_x+D_0)-gD_0\|_g=\|L_x\|=\|x\|=\|g^{-1}Z(1)\|.
$$
Thus if we write $p=g(1)\in\Omega$, it must be $g=U_{p^{1/2}}k$ for some automorphism $k$, and then
$$
\|Z(1)\|_{g1}=\|g^{-1}Z(1)\|=\|k^{-1}U_{p^{-1/2}}Z(1)\|=\|U_{p^{-1/2}}Z(1)\|.\qedhere
$$
\end{proof}

\subsection{Comparing both distances}

Our goal in this section is to prove that $d'=\dist_{\Omega}$ in $\Omega$.

\begin{lema}\label{ddprima}Let $\gamma$ join $x,y\in\Omega$, let $\Lambda$ be any lift of $\gamma$ to $\GO$. Then $\Length_{\Omega}(\gamma)\le \Length_{\GO}(\Lambda)$. Moreover for every $x,y\in\Omega$ we have that $\dist_{\Omega}(x,y) \leq d'(x,y)$.
\end{lema}
\begin{proof}
We compute the speed of $\gamma$ using the previous lemma:
	\begin{equation*}
		\|\dot{\gamma}\|_{\gamma} = \inf_{D \in \der} \|\Lambda'- \Lambda D\|_{\Lambda} \leq \|\Lambda'\|_{\Lambda}.
	\end{equation*}
This implies that $\Length_{\Omega}(\gamma) \leq \Length_{\GO}(\Lambda)$.  Now take $\gamma=e^{\Lambda}$ a path joining $x$ to $y$ in $\Omega$, as the path $\gamma$ and its lift are arbitrary, taking infimum we have that
	\begin{equation*}
		\dist_{\Omega}(x,y) \leq \dist_{\GO}(\Lambda_0,\Lambda_1).
	\end{equation*}
We have to prove that for every pair $k_1,k_2\in \Aut$ we have that $d_{\Omega}(x,y) \leq d_{G(\Omega)}(U_{x^{1/2}}k_1,U_{y^{1/2}}k_2)$. Then, taking infimum  we will have the desired conclusion. If $k_1,k_2$ are in different connected components, $d_{G(\Omega)}(U_{x^{1/2}}k_1,U_{y^{1/2}}k_2) = \infty$, so we can assume they are in the same connected component. Take $k_t$ a smooth path joining $k_1$ and $k_2$ in $\Aut$, $t\in [1,2]$.  
Now we consider the path $\Lambda_t = U_{\gamma_t^{1/2}}k_t$, which is a smooth lift of $\gamma$ to $\GO$, as $k_t(1) = 1$ for every $t$. As shown above,
	\begin{equation*}
		\dist_{\Omega}(x,y) \leq \dist_{\GO}(\Lambda_0,\Lambda_1) = \dist_{\GO}(U_{x^{1/2}}k_1,U_{y^{1/2}}k_2),
	\end{equation*}
and this finishes the proof.
\end{proof}

\begin{defi}[Isometric lifts] Given a piecewise smooth path $\gamma_t\subset\Omega$, we call a piecewise smooth path $\Lambda_t\subset\GO$ an \textit{isometric lift} of $\gamma$ if it is a lift of $\gamma$ and $\Length_{\Omega}(\gamma) = \Length_{\GO}(\Lambda)$. We call it an $\varepsilon-$isometric lift if 
$$
\Length_{\Omega}(\gamma) \leq \Length_{\GO}(\Lambda) + \varepsilon.
$$
\end{defi}

As $q:\GO\to \Omega$ given by $g\mapsto g\cdot p=g(p)$ is a smooth submersion and a quotient map (Remark \ref{cociente}), it was proved in \cite[Theorem 3.25]{largrupos} that every path $\gamma$ has an $\varepsilon-$isometric lift for every positive $\varepsilon$, and this is sufficient to prove the reversed inequality $d'(x,y)\le \dist_{\Omega}(x,y)$ for our distances in $\Omega$. In our particular case, however, we will be able to find an isometric lift, as we have a good decomposition of every element in $\GO$.

\begin{rem}
For $\Gamma$ to be an isometric lift, it has to be a horizontal lift, i.e. a lift for which $\Lambda^{-1}{\Lambda'} \in  \mathbb{L}$. This is because, as we have seen before, $\|L + D\| \geq \|L\|$ for every $D$ derivation and $L\in \mathbb L$, so an horizontal lift always has shorter length.
\end{rem}

\begin{rem}[$V$-identities] Recall the $V$-operators, 
$$
V_{x,y}(z) = U_{x,z}(y)=L_xL_zy+L_zL_xy-L_{xz}y=x(zy)+z(xy)-y(xz),
$$
for $x,y,z\in V$, therefore 
$$
V_{x,y}=L_xL_y-L_yL_x+L_{xy}=[L_x,L_y]+L_{xy}
$$
and then $\overline{V}_{x,y}=[L_x,L_y]-L_{xy}=-V_{y,x}$ by Lemma \ref{overline}. Moreover, we also obtain 
\begin{equation}\label{v+v}
V_{a,b}+V_{b,a}=L_{2ab} \qquad \textrm{ and } \qquad V_{a,b}-V_{b,a}=[L_a,L_b].
\end{equation}
The following identity holds (it can be established in a special Jordan algebra, and then by means of the theorem of McDonald \cite[Section I.9]{jacobson}, it holds on any Jordan algebra):
\begin{equation}\label{uves}
	U_{x^{-1}}U_{x,y}(z) =  V_{x^{-1},y}(z).
\end{equation}
\end{rem}

Take $\Lambda\subset\GO$ a lift of $\gamma\subset\Omega$. Every element $g\in\GO$ can be decomposed as $g = U_xk$, where $x\in \Omega$ and $k\in\Aut$ (see for instance \cite[Corollary 3.28]{larluna1}). Then, using this decomposition for $\Lambda$ and noting that the positive square root of an element is unique, we have that $\Lambda = U_{\gamma^{1/2}}k_t$ where $k_t$ is a path of automorphisms.

\begin{lema}\label{ldeco}
Let $\gamma\subset \Omega$ be smooth, let $\Lambda = U_{\gamma^{1/2}}k_t$ be any smooth lift of $\gamma$, where $k_t\subset \Aut$. Then the horizontal-vertical decomposition $\Lambda^{-1}\Lambda'=L_{H_{\gamma}}+ D_{\gamma}\in \mathbb L\oplus\der$ is given by
$$
H_{\gamma}(t)=2k_t^{-1}(L_{\gamma_t^{-1/2}{(\gamma_t^{1/2})'}})k_t, \qquad D_{\gamma}(t)= k_t^{-1} [L_{\gamma_t^{-1/2}},L_{(\gamma_t^{1/2})'}]  k_t +  k_t^{-1}{k_t'}.
$$
\end{lema}
\begin{proof}
We have  by (\ref{uves})
\begin{align*}
	\Lambda^{-1}{\Lambda'}& = k^{-1} U_{\gamma^{-1/2}}{(U_{\gamma^{1/2}})'}k + k^{-1}{k'} = 2k^{-1} U_{\gamma^{-1/2}}U_{\gamma^{1/2},{(\gamma^{1/2})'}}\,k + k^{-1}{k'}\\
	&=  k^{-1} 2V_{\gamma^{-1/2},{(\gamma^{1/2})'}}k + k^{-1}{k'}.
\end{align*}
Moreover, if we write for $T\in \str$ by means of Lemma \ref{overline} the decomposition $T = \frac{T + \overline{T}}{2} + \frac{T - \overline{T}}{2}$, the first summand is a derivation and the second is an $L$ operator. Then, as for every $V$ operator we have $\overline{V_{x,y}} = -V_{y,x}$ by the previous remark, it follows that
\begin{equation*}
\begin{aligned}
	2V_{\gamma^{-1/2},{(\gamma^{1/2})'}} &= \left(V_{\gamma^{-1/2},{(\gamma^{1/2})'}} + \overline{V_{\gamma^{-1/2},{(\gamma^{1/2})'}}}\right) + \left(V_{\gamma^{-1/2},{(\gamma^{1/2})'}} - \overline{V_{\gamma^{-1/2},{(\gamma^{1/2})'}}}\right)\\
	&= \left(V_{\gamma^{-1/2},{(\gamma^{1/2})'}} - V_{{(\gamma^{1/2})'},\gamma^{-1/2}}\right) + \left(V_{\gamma^{-1/2},{(\gamma^{1/2})'}} + V_{{(\gamma^{1/2})'},\gamma^{-1/2}}\right).
\end{aligned}
\end{equation*}
Now note that for every automorphism $k$ we have that $\overline{kHk^{-1}} =k\overline{H}k^{-1}$. Then
\begin{equation*}
\begin{aligned}
	\Lambda^{-1}{\Lambda'} = \;& k^{-1}\, 2\, V_{\gamma^{-1/2},{(\gamma^{1/2})'}}k + k^{-1}{k'} \\
	=\; & k^{-1} \left(V_{\gamma^{-1/2},{(\gamma^{1/2})'}} + V_{{(\gamma^{1/2})'},\gamma^{-1/2}}\right)k \;+\\
	&+  k^{-1} \left(V_{\gamma^{-1/2},{(\gamma^{1/2})'}} - V_{{(\gamma^{1/2})'},\gamma^{-1/2}}\right)k +  k^{-1}{k'},
\end{aligned}
\end{equation*}
thus
$$
k^{-1} \left(V_{\gamma^{-1/2},{(\gamma^{1/2})'}} + V_{{(\gamma^{1/2})'},\gamma^{-1/2}}\right)k
$$
is the $L$ component of  $\Lambda^{-1}{\Lambda'}$ and 
$$
k^{-1} \left(V_{\gamma^{-1/2},{(\gamma^{1/2})'}} - V_{{(\gamma^{1/2})'},\gamma^{-1/2}}\right)k +  k^{-1}{k'}
$$
is the derivation component of $\Lambda^{-1}{\Lambda}'$. The proof finishes by applying the identites in (\ref{v+v}).
\end{proof}

\begin{prop}\label{ref:isomlift}
	Given $\gamma$ a piecewise smooth path in $\Omega$ joining $x$ to $y$, there exists a unique horizontal lift $\Lambda\subset \GO$ with $\Lambda(0) = U_{x^{1/2}}$. This lift is given by $\Lambda = U_{\gamma^{1/2}}k_t$, where $k_t$ is a path of automorphisms with initial point $Id$, that satisfies the differential equation
	\begin{equation}\label{eq:ecdifislift}
		{k_t'} = [L_{(\gamma_t^{1/2})'}\,,L_{\gamma_t^{-1/2}}] k_t.
	\end{equation}
	Moreover, this lift is isometric.
\end{prop}
\begin{proof}
Note first that the differential equation has a solution inside $\Aut$, since the bracket of two $L$'s is a derivation. Then using Lemma \ref{ldeco}, it is easy to check that  $\Lambda = U_{\gamma^{1/2}}k_t$ is horizontal. Therefore horizontal lifts exist, and moreover they are unique because the solution of the differential equation with initial data $k_0=Id$ is unique. Let us prove now that this lift is isometric. Let $\tilde{\Lambda}$ be any other lift of $\gamma$. As we said before, $\tilde{\Lambda} = U_{\gamma^{1/2}}\tilde{k_t}$ for $\tilde{k_t}$ a smooth path of automorphisms. As for every $x\in\V$ and $D\in\der$  we have that $\|L_x + D\| \geq \|L_x\|$, the norm of $\tilde{\Lambda}^{-1}{\tilde{\Lambda}'}$ is greater than the norm of its $L-$component. Moreover, as automorphisms are isometries,
	\begin{equation*}
	\|\tilde{\Lambda}^{-1}{\tilde{\Lambda}'}\| \geq \| \tilde{k}^{-1} L_{H_{\gamma}}\tilde{k}\| = \| k^{-1} L_{H_{\gamma}}k \| = \|\Lambda^{-1}{\Lambda'}\|.
	\end{equation*}
 Now let $\varepsilon>0$, and let $\tilde{\Lambda}$ be an $\varepsilon$-lift of $\gamma$  as in \cite[Theorem 3.25]{largrupos}, then by Lemma \ref{ddprima} and the previous inequality
 $$
\Length_{\Omega}(\gamma)\le  \Length_{\GO}(\Lambda)\le \Length_{\GO}(\tilde{\Lambda})\le \Length_{\Omega}(\gamma)+\varepsilon
 $$
 Since $\varepsilon$ is arbitrary, it follows that $\Lambda$ is an isometric lift of $\gamma$.
	\end{proof}

\begin{teo}\label{mismad}
	For every $x,y\in \Omega$ we have  $d'(x,y)= \dist_{\Omega}(x,y)$.
\end{teo}
\begin{proof}
The inequality $\dist_{\Omega}\le d'$ was obtained earlier in Lemma \ref{ddprima}, let us prove that the other inequality holds.	 Let $\gamma$ be the geodesic of $\Omega$ joining $x,y$,  let $\Lambda$ be the isometric lift of $\gamma$ that starts at $U_{x^{1/2}}$. We have that $\Lambda$ finishes in $U_{y^{1/2}}k$ for some automorphism $k$. Then
	\begin{equation*}
	\begin{aligned}
		d'(x,y) &=  \dist_{\GO}(U_{x^{1/2}}\Aut,U_{y^{1/2}}) \leq \dist_{\GO}(U_{x^{1/2}}k^{-1},U_{y^{1/2}}) \\
	& \dist_{\GO}(U_{x^{1/2}},U_{y^{1/2}}k) \leq \Length_{\GO}(\Lambda) = \Length_{\Omega}(\gamma) =\dist_{\Omega}(x,y).
	\end{aligned}
	\end{equation*}
Then $d'(x,y) \leq \dist_{\Omega}(x,y)$.
\end{proof}

\smallskip

\subsection{The metric geometry of $\GO$}

We now return the problem of finding short paths for the Finsler metric introduced in $\GO$, with the help of the additional tools developed in the previous section.

\begin{teo}[Isometric lift of a geodesics] Let $\alpha_t=U_{p^{1/2}}e^{tv}$ be a geodesic of $\Omega$. Then the isometric lift $\Lambda\subset \GO$ of $\alpha$ is 
$$
\Lambda_t=U_{p^{1/2}}e^{tL_v}=U_{p^{1/2}}U_{e^{tv/2}}.
$$
\end{teo}
\begin{proof}
Note that $U_{p^{1/2}}e^{tL_v}(1)=U_{p^{1/2}}U_{e^{tv/2}}(1)=U_{p^{1/2}}e^{tv}=\alpha(t)$, therefore $\Lambda_t$ is a lift of $\alpha$. Since $\Lambda_t^{-1}\Lambda_t'=L_v$, we have that $\Lambda$ is horizontal and moreover
$$
\Length_{\GO}(\Lambda)=\int \|\Lambda^{-1}\Lambda'\|=\|L_v\|=\|v\|=\Length_{\Omega}(\alpha)
$$
therefore it is the isometric lift.
\end{proof}

From this we can characterize certain minimizing paths in $\GO$ and compute the distance:

\begin{coro}\label{coroh}
	The geodesic $\Lambda_t=U_pe^{tL_v}$ is minimizing in $\GO$, i.e.
	$$
	\Length_{\GO}(\Lambda)=\|L_v\|=\|v\|=\dist_{\GO}(\Lambda_0,\Lambda_1)=\dist_{\GO}(U_p,U_pe^{L_v}).
	$$
\end{coro}
\begin{proof}
By the previous theorem, $\Lambda$ is an isometric lift of the minimizing geodesic $\alpha(t)=U_pe^{tv}$ in $\Omega$.	Then if $\Phi$ is any path in $\GO$ joining the same endpoints than $\Lambda$, the path $\phi=\Phi(1)$ is a path in $\Omega$ joining the same endpoints than $\alpha$. Therefore by Lemma \ref{ddprima} and the fact that $\alpha$ is minimizing in $\Omega$, we have
$$
\Length_{\GO}(\Lambda)=\Length_{\Omega}(\alpha)\le \Length_{\Omega}(\phi)\le \Length_{\GO}(\Phi),
$$
therefore $\Lambda$ is minimizing in $\GO$.
\end{proof}

This enables the following remark:

\begin{rem}[Distance from $Id$ to the fiber $e^{L_z}\Aut$] 
Every $g\in\GO$ can be written as $g=U_xk$ for some $x\in\Omega$ and $k\in\Aut$, the metric is left-invariant and $\Aut$-right-invariant, and the metric in $\Omega$ is the quotient metric. Therefore we can conclude, by the previous corollary, that for every $g\in \GO$ and $z\in\V$ we have
\begin{align*}
\dist_{\GO}(g,ge^{L_z})& =\dist_{\GO}(1,e^{L_z})=\|z\|=\dist_{\GO}(1,e^{L_z}\Aut)
\end{align*}
and $t\mapsto e^{tL_z}$ is the optimal path between $Id\in\GO$ and the fiber $e^{L_z}\Aut\subset\GO$.
\end{rem}

\begin{rem}Let $D\in \der$,  let $\delta=-iD^{\mathbb C}$ where $D^{\mathbb C}$ is the complexification of $D$, i.e. $D^{\mathbb C}(x+iy)=Dx+iDy$. Then it is easy to check that $\delta$ is a derivation in $\V^{\mathbb C}$. Moreover, it is a $*$-derivation i.e. $\delta(w^*)=-(\delta(w))^*$, where $(x+iy)^*=x-iy$ for $x,y\in V$ is the usual involution of the complexification. Then by \cite[Corollary 10]{youngson2}, we have that $\delta\in Her \B(V^{\mathbb C})$ that is $\varphi(\delta)\subset\mathbb R$ for all $\varphi\in (\B(V^{\mathbb C}))^*$ such that $\varphi(Id)=1=\|\varphi\|$.
\end{rem}

We already noted (and used repeatedly) that $\|L_x+D\|\ge \|L_x\|$, thus the derivations are in some sense orthogonal to $\mathbb L$ (this is the notion of Birkhoff orthogonality in normed spaces). Now we show that the $L$ operators are orthogonal to $\der$ i.e.:

\begin{prop}Let $x\in \V$, let $D\in\der$. Then $\|L_x+D\|\ge \|D\|$.
\end{prop}
\begin{proof}
Let $\varphi$ be a bounded linear functional in the dual of $\B(V^{\mathbb C})^*$ such that $\varphi(Id)=1=\|\varphi\|$. Since $|z|\ge \pm \mathrm{Im\,}z$ for any $z\in \mathbb C$, we have
$$
\|L_x+D\|\ge |\varphi(L_x+D)|\ge \pm \mathrm{Im\,}\varphi(L_x)\pm \mathrm{Im\,} i\varphi (\delta),
$$
where $\delta=-iD$ as in the previous remark. Now the operator $L_x$ is Hermitian, therefore $\varphi(L_x)\subset \mathbb R$, and so is $\delta$, hence $\|L_x+D\|\ge 0\pm \varphi (\delta)$. Thus the numerical range of $\delta$ is in the interval bounded by $\|L_x+D\|$, i.e.
$$
V(\delta)=\{\varphi(\delta): \|\varphi\|=1=\varphi(Id)\}\subset [-\|L_x+D\|, \|L_x+D\|].
$$
Let $\textrm{co}(\mathcal{C})$ denote the closed convex hull of the set $\mathcal{C}\subset\mathbb R$. Since $\delta$ is Hermitian, we have $\textrm{co}(\sigma(\delta))=V(\delta)$ and $\|\delta\|=r(\delta)=\max \{\lambda: \lambda\in V(\delta)\}\le \|L_x+D\|$, see \cite[$\S$10]{bonsall} for details. Therefore $\|D\|=\|\delta\|\le \|L_x+D\|$.
\end{proof}

We know by Theorem \ref{mini} that the one-parameter groups $t\mapsto e^{tD}$ are optimal with respect to any other path $\Lambda\subset \Aut$ joining its endpoints. Since derivations are also in good position with respect to $L$-operators, the question arises: is the one-parameter group in $\Aut$ optimal with respect to any other path $\Lambda\subset \GO$ joining its endpoints?

\section*{Acknowledgments} 

This research was supported by Universidad de Buenos Aires, Agencia Nacional de Promoci\'on de Ciencia y Tecnolog\'\i a (ANPCyT-Argentina) and Consejo Nacional de Investigaciones Cient\'\i ficas y T\'ecnicas (CONICET-Argentina). This line of research on Jordan Banach algebras was encouraged by the talks and informal discussions held by G. Larotonda with Cho-Ho Chu, Bas Lemmens, Jimmy Lawson, Yongdo Lim, Karl-Hermann Neeb and Harald Upmeier among others, during two workshops on Jordan Algebras and Convex Cones (Leiden 2017 and Jeju 2019).

\renewcommand{\baselinestretch}{1.25}


\begin{thebibliography}{XX}

\bibitem{alrv} E. Andruchow, G. Larotonda, L. Recht, A. Varela, \textit{The left invariant metric in the general linear group}. J. Geom. Phys. 86 (2014), 241--257.

\bibitem{arnold} V. Arnol'd, \textit{Sur la g\'eom\'etrie diff\'erentielle des groupes de Lie de dimension infinie et ses applications \'a l'hydrodynamique des fluides parfaits}. (French) Ann. Inst. Fourier (Grenoble) 16 (1966), fasc. 1, 319--361.

\bibitem{beltita} D. Beltita, \textit{Smooth homogeneous structures in operator theory}. Chapman \& Hall/CRC Monographs and Surveys in Pure and Applied Mathematics, 137. Chapman \& Hall/CRC, Boca Raton, FL, 2006.


\bibitem{bhatia} R. Bhatia,  \textit{Matrix analysis}. Graduate Texts in Mathematics, 169. Springer-Verlag, New York, 1997.


\bibitem{bonsall} F. F. Bonsall, J. Duncan, \textit{Complete normed algebras} (Berlin, Heidelberg, New York, Springer-Verlag, 1973.


\bibitem{bourbaki} N. Bourbaki, \textit{Groupes et alg\'ebres de Lie}, Chapitres 2 et 3, Masson, Paris, 1990.

\bibitem{conde} C. Conde, \textit{Differential geometry for nuclear positive operators}. Integral Equations Operator Theory 57 (2007), no. 4, 451--471.

\bibitem{corach} G. Corach, H. Porta, L. Recht,  \textit{The geometry of the space of selfadjoint invertible elements in a $C^*$-algebra}. Integral Equations Operator Theory 16 (1993).

\bibitem{corach2} G. Corach, H. Porta, L. Recht,  \textit{Geodesics and operator means in the space of positive operators}. Internat. J. Math. 4 (1993), no. 2, 193--202.

\bibitem{maes} G. Corach, A. Maestripieri, D. Stojanoff, \textit{Orbits of positive operators from a differentiable viewpoint}. Positivity 8 (2004), no. 1, 31--48. 


\bibitem{chu} C.-H. Chu, \textit{Infinite dimensional Jordan algebras and symmetric cones}. J. Algebra 491 (2017), 357--371.

\bibitem{chucrelle} C.-H. Chu, \textit{Siegel domains over Finsler symmetric cones}. J. Reine Angew. Math. 778 (2021), 145--169.

\bibitem{faraut} J. Faraut, A. Kor\'anyi, \textit{Analysis on symmetric cones}. Oxford Mathematical Monographs. Oxford Science Publications. The Clarendon Press, Oxford University Press, New York, 1994.

\bibitem{stormer} H. Hanche-Olsen, E. Stormer,  \textit{Jordan operator algebras}. Monographs and Studies in Mathematics, 21. Pitman (Advanced Publishing Program), Boston, MA, 1984.

\bibitem{jacobson} N. Jacobson, \textit{Structure and representations of Jordan algebras}. American Mathematical Society Colloquium Publications, Vol. XXXIX American Mathematical Society, Providence, R.I. 1968.

\bibitem{kaupmeier} W. Kaup, H. Upmeier \textit{Jordan algebras and symmetric Siegel domains in Banach spaces}, Math. Z. 157 (1977), 179--200.


\bibitem{lang} S. Lang, \textit{Differential and Riemannian manifolds}. Third edition. Graduate Texts in Mathematics, 160. Springer-Verlag, New York, 1995.

\bibitem{larestr} G. Larotonda,   \href{https://mate.dm.uba.ar/~glaroton/estructuras.pdf}{Estructuras Geom\'etricas para las Variedades de Banach}. Universidad Nacional de General Sarmiento, colecci\'on ``Ciencia, Innovaci\'on y Tecnolog\'\i a'', 2012.

\bibitem{larhs} G. Larotonda, \textit{Nonpositive curvature: a geometrical approach to Hilbert-Schmidt operators}. Differential Geom. Appl. 25 (2007), no. 6, 679--700.

\bibitem{larineq} G. Larotonda, \textit{Norm inequalities in operator ideals}. J. Funct. Anal. 255 (2008), no. 11, 3208--3228.

\bibitem{largrupos} G. Larotonda, \textit{Metric geometry of infinite-dimensional Lie groups and their homogeneous spaces}. Forum Math. 31 (2019), no. 6, 1567--1605.


\bibitem{larluna1} G. Larotonda, J. Luna, \textit{On the structure group of an infinite dimensional JB-algebra}. J. Algebra 622 (2023), 366--403.

\bibitem{ll} J. Lawson, Y. Lim, \textit{Metric convexity of symmetric cones}. Osaka J. Math. 44 (2007), no. 4, 795-816.

\bibitem{bas} B. Lemmens, M. Roelands, \textit{Unique geodesics for Thompson's metric}. Ann. Inst. Fourier (Grenoble) 65 (2015), no. 1, 315--348. 

\bibitem{bas2} B. Lemmens, M Roelands, M. Wortel, \textit{Hilbert and Thompson isometries on cones in JB-algebras}. Math. Z. 292 (2019), no. 3-4, 1511--1547.

\bibitem{loos} O. Loos, \textit{Symmetric Spaces I: General Theory}, Benjamin, New York, 1969.


\bibitem{martin} M. Mart\'in,  \textit{On different definitions of numerical range}. J. Math. Anal. Appl. 433, 2016.

\bibitem{mccrimmon} K. McCrimmon, \textit{A taste of Jordan algebras}. Universitext. Springer-Verlag, New York, 2004.



\bibitem{neeb} K.-H. Neeb, \textit{A Cartan-Hadamard theorem for Banach-Finsler manifolds}. Geom. Dedicata 95 (2002), 115--156.

\bibitem{nussbaum} R. Nussbaum, \textit{Finsler structures for the part metric and Hilbert's projective metric and applications to ordinary diffeential equations}, Diff. Integ. Eq. 7 (1994), 1649--1707.

\bibitem{upmeier} H. Upmeier,  \textit{Symmetric Banach manifolds and Jordan C$^*$-algebras}, North Holland, Amsterdam, 1985


\bibitem{wright2} J.D. M. Wright, \textit{Jordan $C^*$-algebras}. Michigan Math. J. 24 (1977), no. 3, 291-302.


\bibitem{youngson2} M. A. Youngson, \textit{Hermitian operators on Banach Jordan algebras}. Proc. Edinburgh Math. Soc. (2) 22 (1979), no. 2, 169--180.

\end{thebibliography}
\end{document}